\documentclass[11pt]{article}

\usepackage{multirow}
\usepackage{algorithmic}
\usepackage{amsmath}
\usepackage{amsfonts}
\usepackage{amsthm}
\usepackage[english]{babel} 
\usepackage{dsfont}
\usepackage{amssymb}
\usepackage{graphicx}        
\usepackage{cite}                            
\usepackage{subcaption}
\usepackage{epstopdf}
\usepackage{epsfig}
\usepackage{subfloat}
\usepackage{float}
\usepackage[thinlines]{easytable}
\usepackage{array}
\usepackage{mathtools}
\usepackage{stmaryrd}
\usepackage{hyperref}
\usepackage{comment}
\usepackage{bm}
\usepackage[shortlabels]{enumitem}
\usepackage{epstopdf}
\hypersetup{
    colorlinks=true,
    linkcolor=blue,
    filecolor=magenta,      
    urlcolor=cyan,
    pdftitle={Overleaf Example},
    pdfpagemode=FullScreen,
    }

\usepackage{booktabs}
\usepackage{url}
\usepackage[font=small,labelfont=bf,textfont=it,indention=.1cm,width=1.1\textwidth]{caption}
\usepackage[capitalise]{cleveref}

\allowdisplaybreaks
\numberwithin{equation}{section}

\newtheorem{remark}{Remark}[section]

\newtheorem{lemma}{Lemma}[section]
\newtheorem{theorem}{Theorem}[section]
\newtheorem{corollary}{Corollary}[section]
\newtheorem{proposition}{Proposition}[section]
\usepackage[linesnumbered,algoruled,boxed,lined]{algorithm2e}
\newenvironment{myfont}{\fontfamily{lmtt}\selectfont}{\par}
\DeclareTextFontCommand{\textmyfont}{\myfont}
\newcommand{\splitcell}[2][c]{%
  \begin{tabular}[c]{@{}c@{}}\strut#2\strut\end{tabular}%
}

\def\RR{\mathbb R}

\def\ve{\varepsilon}

\def\F{\mathcal{F}}

\newcommand{\err}{\textup{Err}}
\newcommand{\dashdownarrow}{\rotatebox{-90}{$\dashrightarrow$}}

\renewcommand{\hat}{\widehat}
\renewcommand{\bar}{\overline}

\newcommand{\gb}[1]{{\color{black}#1}}

\def\be{\begin{equation}}
\def\ee{\end{equation}}
\def\bea{\begin{eqnarray}}
\def\eea{\end{eqnarray}}

\def \d {\text{d}}
\def \var {\textup{var}}
\def \red {\textup{sel}}

\newtheorem{assumption}{Assumption}[section]
\title{Consensus based optimization with memory effects: \\ random selection and applications}

\author{Giacomo Borghi\footnote{RWTH Aachen University, Institute for Geometry and Applied Mathematics, Aachen, Germany (borghi@eddy.rwth-aachen.de)}
\and Sara Grassi\footnote{University of Ferrara, Department of Mathematics and Computer Science \& Center for Modelling Computing and Statistics, Ferrara, Italy (sara.grassi@unife.it, lorenzo.pareschi@unife.it)}
\and Lorenzo Pareschi\footnotemark[2]}

\begin{document}
\maketitle

\begin{abstract}
In this work we extend the class of Consensus-Based Optimization (CBO) metaheuristic methods by considering memory effects and a random selection strategy. The proposed algorithm iteratively updates a population of particles according to a consensus dynamics inspired by social interactions among individuals. The consensus point is computed taking into account the past positions of all particles. While sharing features with the popular Particle Swarm Optimization (PSO) method, the exploratory behavior is fundamentally different and allows better control over the convergence of the particle system. 
We discuss some implementation aspects which lead to an increased efficiency while preserving the success rate in the optimization process. In particular, we show how employing a random selection strategy to discard particles during the computation improves the overall performance. Several benchmark problems and applications to image segmentation and Neural Networks training are used to validate and test the proposed method.
A theoretical analysis allows to recover convergence guarantees under mild assumptions on the objective function. This is done by first approximating the particles evolution with a continuous-in-time dynamics, and then by taking the mean-field limit of such dynamics. Convergence to a global minimizer is finally proved at the mean-field level.

\end{abstract}


\tableofcontents

\section{Introduction}

Meta-heuristic algorithms are recognized as trustworthy, easy to understand optimization methods which have been widely applied to several fields such as Machine Learning \cite{Li2022}, path planning  \cite{Liang2020} and image processing \cite{tuli2022}, to name a few. Starting from a set of possible solutions, a meta-heuristic algorithm typically updates such set iteratively by combining deterministic and stochastic choices, often inspired by natural phenomena. 
Exploration of the search space and exploitation of the current knowledge are the two fundamental mechanisms driving the algorithm iteration \cite{crepinsek2013explore}.
Examples of established meta-heuristic algorithms are given by Genetic Algorithm (GA) \cite{holland1992adaptation,tang1996genetic}, Simulated Annealing (SA) \cite{kirkpatrick1983sa}, Particle Swarm Optimization (PSO) \cite{kennedy1995particle} and Differential Evolution (DE) \cite{storn1997differential}. We refer to \cite{hussain2019review} for a complete literature review.

Consensus-Based Optimization (CBO) is a class of gradient-free meta-heuristic algorithms inspired by consensus dynamics among individuals. After its introduction \cite{pinnau2017consensus} it has gained popularity among the mathematical community due to its robust mathematical framework \cite{carrillo2018analytical,huang2021meanfield,ha2021timediscrete,fornasier2021consensusbased}. In CBO algorithms, a population of particles concentrates around a consensus point given by a weighted average of the particles position. In the computation of such consensus point, more importance is given to those particles attaining relatively low values of the objective function \gb{by means of the Gibbs distribution}. The exploration mechanism is introduced by randomly perturbing the particles positions at each iteration. Particles which are close to the consensus point are subject to small perturbations, while those that are far from it display a more exploratory behavior. 

In this work, following the recent analysis in \cite{grassi2021from}, we study a Consensus-Based Optimization algorithm with Memory Effects (CBO-ME) where the consensus point is computed among the whole history of the particles positions and not just \gb{among} the positions of the current iteration, as in the original CBO method. This is done by keeping track of the best position found so far by each particle, and \gb{by} computing the consensus point among these ``personal'' bests. While sharing common elements with PSO, such as \gb{the convergence mechanism} to a promising point and the presence of personal bests, CBO-ME differs in the way the exploration mechanism is implemented. Indeed, in CBO-ME, as in CBO algorithms, the stochastic behavior is given by adding Gaussian noise to the particles dynamics and can be tuned independently on the exploitation mechanisms, leading to a better control over the particles convergence.
Therefore, while in classical PSO methods it is the balance between local best and global best that governs the optimization strategy, in CBO methods it is the balance between exploration and exploitation mechanisms that determines the choice of parameters. We recall that a generalization of PSO methods that allows leveraging the same flexibility in searching the global minimum as in CBO algorithms has been recently presented in \cite{grassi2021from}.

Many real-life problems, especially those regarding Machine Learning, require to optimize a large number of parameters. Therefore, it essential to design fast algorithms to save computational time and memory. This is a major weakness of swarm-based methods, which require a set of particles to minimize the problem, unlike gradient-based methods that can work on a single particle trajectory. 
For methods based on a collection of particles, existing algorithms can be improved by discarding particles whenever the system has a prominent exploitative behavior. This is sometimes referred as ``natural selection strategy'' in the DE literature \cite{storn1997differential,changhe2007} and aims to discard the non-promising solutions. Inspired by particle simulations techniques where it is important to preserve the particles \gb{probability} distribution, we examine a ``random selection strategy'' where particles are discarded randomly based on the local consensus achieved.
We will discuss such implementation aspects by testing CBO-ME against high-dimensional learning problems and theoretically analyze the impact of the random selection strategy on the system. In particular, we prove that if the full particle system is expected to converge towards a solution to \gb{the minimization problem}, so will the reduce one, provided a sufficient number of particles remains active. 
Note that, such analysis can be generalized to other particle dynamics and may be of independent interest.  

Owing to the convergence analysis of CBO algorithms \cite{carrillo2018analytical,huang2021meanfield,fornasier2022aniso,fornasier2021consensusbased} and recent analysis of PSO \cite{grassi2021from,huang2022PSO} we are able to prove convergence of the algorithm under mild assumption on the objective function. This is done by first approximating the algorithm with a continuous-in-time dynamics and secondly by giving a probabilistic description to the particles system. By assuming propagation of chaos \cite{sznitman1991chaos}, particles are considered to behave independently according to the same law. This allows to reduce the possible large system of equations to a single partial differential equation: the so-called mean-field model. Such model is then analyzed to recover convergence guarantees under precise assumption on the objective function. Developed in the field of statistical physics, this approach has shown be fruitful in studying particle-based meta-heuristic algorithms \cite{fornasier2022aniso,fornasier2021consensusbased,huang2022PSO}. 
We note that convergence in mean-field law was recently proved in \cite{riedl2022leveraging} in an independent work. 

The rest of the paper is organized as follows. Section \ref{sec:2} is devoted to the introduction of the CBO-ME algorithm with random selection and comparison with CBO methods without memory effects as well as PSO.  In Section \ref{sec:numerics}, 
validate the proposed method against several benchmark problems and two Machine Learning tasks. 
Theoretical convergence guarantees and analysis of the random selection strategy are summarized in Section \ref{sec:convergence}.
Some final remarks are given in Section \ref{sec:conclusions}. Technical details of the theoretical analysis are given in Appendix \ref{app:convergence}.

\section{Consensus-based optimization with memory effects}
\label{sec:2}

In this section, we present the Consensus-Based Optimization algorithm with Memory Effects (CBO-ME)
to solve problems of the form
\be
x^* \in \underset{x\in \RR^\d}{\textup{argmin}}\, \F(x) \,,
\label{eq:pb}
\ee
where $\RR^\d,\, \d \in \mathbb{N}$ is the, possibly large, search domain for the continuous function $\F \in C(\RR^\d,\RR)$. We will do so by \gb{also} highlighting similarities and differences between classical CBO methods and PSO algorithms.

\subsection{Particles update rule}
\label{sec:updaterule}
At each iteration step $k$ and for every particle $i=1, \dots, N$, we store its position $x^k_i$ and its best position found so far $y_i^k, \gb{\F(y_i^k) = \min_{h \leq k} \F(x_i^h)}$. The best positions are used to compute a consensus point 
\be
\bar y^{\alpha,k} = \sum_{i=1}^N \,\omega_i^k\, y^k_i\, \quad \textup{with}\quad \omega_i^k = \frac{e^{-\alpha \F(y^k_i)}}{\sum_{j=1}^N e^{-\alpha \F(y^k_j)}}
\label{eq:yalpha}
\ee
which approximate\gb{s} the global best solution \gb{$\bar y^{\infty,k}$} among all particles and all times for $\alpha>1$. Indeed, thanks to the choice of the weights $\omega_i^k$, we have that 
\be \notag
\bar y^{\alpha,k} \quad \longrightarrow \quad \bar y^{\infty,k} := \textup{argmin}\{\F(y_1^k), \dots,\F(y_N^k) \}
\ee
as $\alpha \to \infty$, provided that there is only one global best position among $\{y_1^k, \dots, y_N^k\}$. Such approximation was first introduced for CBO methods \cite{pinnau2017consensus} as it leads to more amenable theoretical analysis, but it also allows for  more flexibility. Indeed, relatively small values of $\alpha$ \gb{can be} 
 used at the beginning of the computation to promote exploration. Large values of $\alpha$, on the other hand, lead to better exploitation of the computed solutions and to higher accuracy. We note that the weights used in \eqref{eq:yalpha} correspond in statistical mechanics to the Boltzmann-Gibbs distribution associated with the energy $\F$. In this context, $\alpha$ plays the role of the inverse of the system temperature $T$ and the limit $\alpha \to \infty$ corresponds to $T \to 0$.

Once the consensus point $\bar y^{\alpha,k}$ is computed, the particle positions are then updated according to the law
\be
x^{k+1}_i =  x_i^{k} +  \lambda \left( \bar y^{\alpha,k} - x_i^k \right) + \sigma \left( \bar y^{\alpha,k} - x_i^k\right) \otimes \theta^k_{i}
\label{eq:updateCBOME}
\ee
with $\theta^k_{i} \in \RR^d$ randomly sampled from the normal distribution ($\theta^k_i \sim \mathcal{N}(0, \mathbf{I}_\d)$) and where $\otimes$ is the component-wise product.

The update rule is characterized by a deterministic component of strength $\lambda \in (0,1)$ promoting concentration around the consensus point $\bar y^{\alpha,k}$ and a stochastic component of strength $\sigma> 0$ promoting exploration of the search space. As the latter depends on the difference $(\bar y^{\alpha,k} - x_i^k)$, the random behavior is stronger for particles which are far from the consensus point, whereas it is weaker for those that are close to it. 
Also, such exploration resemble\gb{s} an anisotropic diffusive behavior 
\gb{in which every coordinate direction is explored} 
at a different rate. This approach was first proposed in \cite{carrillo2019consensus} in the context of CBO methods and has been proved to suffer less from the curse of dimensionality with the respect to the originally proposed isotropic diffusion given by $\sigma \|\bar y^{\alpha,k} - x_i^k\|_2 \theta_i^k$ with $\theta_i^k$ being again a normally distributed $\d$-dimensional vector \cite{carrillo2019consensus}.

\subsection{Random selection strategy}
\label{sec:selection}
When the particle system concentrates around the consensus point, showing a mostly exploitative behavior, we employ a particle selection strategy. Discarding particles introduces additional stochasticity to the system, while reducing the computational cost. Following the approach suggested in \cite{fhps20-2}, we check the evolution of the system variance to decide how many particles to (eventually) discard. 

For a given set of particles $\mathbf{z} = \{z_i\}_{i \in J}$, the system variance is given by
\be
\var (\mathbf{z}) := \frac 1 {|J|} \sum_{j \in J} \left \| z_j - \textup{m}(\mathbf{z}) \right\|_2^2\, \quad \textup{with}\quad \textup{m}(\textbf{z}):= \frac1{|J|}\sum_{i\in J} z_i \,,
\label{eq:variance}
\ee
where $|J|$ indicates the cardinality of $J$, that is, the number of particles in this context. 

Let $I_k \subseteq  \{1, \dots, N\}$ be the set of active particles at step $k$ and $N_k = |I_k|$. To decide how many particles to select, we compare the variance of the particle system before the position update \eqref{eq:updateCBOME}, $\mathbf{x}^k = \gb{\{x_i^{k} \}_{i\in I_k}}$ and after it, $\tilde{\mathbf{x}}^{k+1} = \{x_{i}^{k+1} \}_{i \in I_k}$. Then, the number $N_{k+1}$ of particles we select for the next iteration is given by 
\be
\begin{split}
\tilde N_{k+1} &= \left \lfloor  N_k \left ( 1+ \mu\, \frac{ \var(\tilde{\mathbf{x}}^{k+1})-\var(\mathbf{x}^{k+1})}{\var(\mathbf{x}^{k+1})}  \right) \right \rfloor \\
 N_{k+1} &=  \min \Big\{ \max\big\{\tilde N_{k+1}, N_{\min}\big \}, N_k \Big \}
\end{split}
\label{eq:reduction}
\ee
%
\gb{with} $\lfloor z \rfloor$ being the integer part of a number $z$ and $N_{\min}\in \mathbb{N}$ the smallest amount of particles we allow to have. \gb{If $N_{k+1} < N_k$}, a subset $I_{k+1} \subset I_k, |I_{k+1}| = N_{k+1}$, of particles is randomly selected to continue the computation.
The parameter $\mu\in [0,1]$ regulates the mechanism: for $\mu=0$ there is no particle discarding, while for $\mu=1$ the maximum number of particles is discarded if the variance is decreasing. As we will see in Section \ref{sec:numerics}, this random selection mechanism 
reduces the computational time without affecting the algorithm performance. We will also theoretically analyze this aspect in Section \ref{s:reductionanalysis}, where we show that convergence properties are preserved.

As stopping criterion, we keep a counter $n$ on how many times $\|\bar y^{\alpha,k+1}-\bar y^{\alpha,k} \|_2$ is smaller than a certain tolerance $\delta_{\textup{stall}}>0.$ If this happens for more than a given $n_{\textup{stall}}$ number of times in a row, we assume the particle system has found a solution and stop the computation. A maximum number of iteration $k_{\max}$ representing the computational budget is also given. The proposed CBO-ME is summarized in Algorithm \ref{alg}.

\begin{remark} 
In the meta-heuristic literature, particles are usually discarded depending on their objective value, in a way that particles with high \gb{objective value} are more likely to be discarded \cite{storn1997differential,changhe2007}. 
The proposed strategy does not add a further heuristic strategy but simply cut down the algorithm complexity.  Also, convergence properties are in this way expected to be preserved. We note that, on the other hand, there is no straightforward way to \gb{both} generate particles and preserve the particle system distribution \gb{at the same time}.
\end{remark}

\begin{algorithm}[h!]
    \SetKwFunction{isOddNumber}{isOddNumber}
    \SetKwInOut{KwIn}{Input}
    \KwIn{$\F, N_0,\gb{N_{\min}},k_{\max},\lambda,\sigma,\alpha, n_{\textup{stall}}$ and $\delta_{\textup{stall}}$\; }
	Inizialize $N_0$ particle positions $x_0^i, i=1, \dots, N$\;
	$y_i^0 \gets x_i^0$ for all $i=1, \dots, \gb{N_0}$\;
	Compute $\bar y^{\alpha,0}$ according to \eqref{eq:yalpha}\;
 	$k \leftarrow 0$, $n \leftarrow 0$\; 
	\While{$k < k_{\max} $ \textup{and}  $ n < n_{\textup{stall}}$}{
		\For{$i = 1$ \KwTo $N_k$}{
			$\theta^k_i \sim \mathcal{N}(0,\mathbf{I}_{\d})$\;
			Compute $x_i^{k+1}$ according to \eqref{eq:updateCBOME}\;
			\uIf{$\F(x_i^{k+1}) < \F(y_i^k)$}
		{$y_i^{k+1} \gets x_i^{k+1}$\;}
				\Else{ $y_i^{k+1} \gets y_i^{k}$\;}
		}
		Compute  $\bar y^{\alpha,k+1}$ according to \eqref{eq:yalpha}\;
		\uIf{$ \|\bar y^{\alpha,k+1} - \bar y^{\alpha,k}\|_2 < \delta_{\textup{stall}}$}
		{$n \leftarrow n+1$\;}
		\Else{ $ n \leftarrow 0$\;}
	       Compute $N_{k+1}$ according to \eqref{eq:reduction}\;
		\uIf{$N_{k+1}<N_k$}
		{Randomly discard $N_{k+1}-N_k$ particles\;}
		$k \leftarrow k+1 $\;
}
    \KwRet{$\bar y^{\alpha,k},\F(\bar y^{\alpha,k})$}
    \caption{Consensus-Based Optimization with Memory Effects (CBO-ME) }
    \label{alg}
\end{algorithm}

\subsection{Comparison with CBO and PSO}

What distinguishes CBO-ME from plain CBO, see e.g \cite{pinnau2017consensus,carrillo2019consensus}, is clearly the introduction of the best positions $\{y^k_i\}_{i=1}^N$ and the fact that the consensus point is calculated among them and not just among the particle positions $\{x^k_i\}_{i=1}^N$ at that given time $k$. Indeed, the classical CBO update rule without memory effects (and with anisotropic diffusion and projection step) is given by 
\be
x^{k+1}_i =  x_i^{k} + \lambda \left( \bar x^{\alpha,k} - x_i^k \right) + \sigma \left( \bar x^{\alpha,k} - x_i^k\right) \otimes \theta^k_i 
\label{eq:updateCBO}
\ee
where $\bar x^{\alpha,k}$ is defined consistently with \eqref{eq:yalpha} (by substituting $y_i^k$ with $x^k_i$).  As we will see in the numerical tests, the use of memory effects improves the algorithm performance. 

Since alignment towards personal bests $y^k_i$ and towards the global best $\bar y^{\infty,k}$ are also the fundamental building blocks of PSO algorithms, we highlight now the main differences and similarities between PSO and CBO-ME. For completeness, we recall the canonical PSO method, see e.g. \cite{poli2007particle}, using the notation of \eqref{eq:updateCBOME} for easier comparison
\be
\begin{cases}
x^{k+1}_i &=  x^{k}_i + v_i^{k+1} \\
v^{k+1}_i &= w v_i^{k} + C_1\left(y^k_i - x_i^k \right) \otimes\hat \theta_{i,1}^k + C_2\left(\bar y^{\infty,k} - x_i^k\right)\otimes\hat \theta_{i,2}^k
\label{eq:updatePSO}
\end{cases}
\ee
where $v^k_i$ are the particles velocities, $w,C_1,C_2>0$ are the algorithm parameters and $ \theta_{i,1}^k, \theta_{i,2}^k$ are uniformly sampled from $[0,1]^\d$, ($\hat \theta_{i,1}^k, \hat \theta_{i,2}^k) \sim \textup{Unif}([0,1]^\d)$. Several variants and improvements have been proposed starting from the above dynamics, but a complete review is beyond the scope of this paper and we refer to the recent survey \cite{wang2018particle} for more references. 

We are interested in highlighting the main differences between \eqref{eq:updateCBOME} and \eqref{eq:updatePSO} regarding the stochastic components: in CBO-ME deterministic and stochastic steps are de-coupled and tuned by two different parameters ($\lambda$ and $\sigma$), while in PSO they are coupled.
Indeed, in \eqref{eq:updatePSO}, deterministic and stochastic components are both controlled by the same parameter: $C_1$ in the case of personal best dynamics and $C_2$ for the global best one. By splitting the term $C_2\left(\bar y^{\infty,k} - x_i^k\right)\hat \theta_{i,2}^k$ into a deterministic step and a zero-mean term we obtain
\be
C_2\left(\bar y^{\infty,k} - x_i^k\right)\otimes \hat\theta_{i,2}^k = \frac{C_2}2 \left(\bar y^{\infty,k} - x_i^k\right) + \frac{C_2}2\left(\bar y^{\infty,k} - x_i^k\right)\otimes \theta_{i,2}^k
\ee
with $\theta^k_{i,2} = 2\hat \theta_{i,2}^k -1,\,\theta_{i,2}^k \sim \textup{Unif}([-1,1]^\d)$. Suggested in \cite{grassi2021from}, such rewriting highlights how increasing the alignment strength towards the global best (by increasing $C_2$) necessary increases the stochasticity of the system as well. In \eqref{eq:updateCBOME} and \eqref{eq:updateCBO}, on the other hand, one is allowed to tune the exploration and exploitation behaviors separately, by either changing \gb{parameter} $\lambda$ or $\sigma$. 


Clearly, CBO-ME also differs from PSO due to its first-order dynamics. Having the aim of resembling birds flocking, the first PSO algorithm \cite{kennedy1995particle} was proposed as a second-order dynamics. 
The inertia weight $w$, introduced later in \cite{shi1998modified}, became an essential parameter to prevent early convergence of the swarm and to increase the global exploration behavior, especially at the beginning of the computation, see e.g.\cite{shi1998modified,nick2011novel} and reviews \cite{poli2007particle,wang2018particle,pso2021review} for more references. 
We note that several other strategies have proposed to improve PSO exploration behavior, see, for example, \cite{zhang2023particle}.
As already mentioned, in CBO methods convergence and exploration are de-coupled and can be tuned separately. Therefore, to keep the algorithm more amenable to theoretical analysis, we consider a simpler first-order dynamics. We note that a CBO dynamics with inertia mechanism was proposed in \cite{jin2022adaptive}.

Similarly, we found the contribution given by the personal best alignment non-essential and difficult to tune. Thus, the lack of alignment towards personal best in \eqref{eq:updateCBOME}. Replacing alignment towards personal best with \gb{Gaussian} noise was also suggested in \cite{yang2020nature} where authors proposed the Accelerated PSO (APSO) algorithm. Further studied in \cite{yang2011acc,gandomi2013}, APSO also allows to de-couple the stochastic component from the deterministic one and the noise is heuristically tuned to decrease during the computation as in Simulated Annealing \cite{kirkpatrick1983sa}. In CBO methods, the noise strength automatically adapts as it depends on the distance from the consensus point, which is also different for every particle. 
For completeness, we note that many other variants of PSO have been proposed to include different explorative behaviors, see e.g. Chaotic PSO \cite{liu2005chaospso}.


%
%

\section{Numerical results}
\label{sec:numerics}
Having discussed the fundamental features of the CBO dynamics with memory effects, we now validate Algorithm \ref{alg} and compare its performance with plain CBO and \gb{PSO}. We will test the methods against several benchmark optimization problems and  analyze the impact of the random selection technique on the convergence speed.
We also employ \gb{Algorithm} \ref{alg} to solve problems arising from applications, such as image segmentation and training of \gb{Machine Learning} architectures for function approximation and image \gb{classification}.

\subsection{Tests on benchmark problems}
\label{sec:benchmark}

We test the proposed algorithm against different optimization problems, by considering 8 benchmark objective functions, see e.g. \cite{JamilY13}, which we report in Table \ref{tab:benchmarks} for completeness. The search space dimension is set to $\d = 20$ and the location of the global best $x^*$ is known.

\begin{table}
\begin{center}
  \centering\scriptsize{
  \begin{tabular}{  | c | c | c | c | c |  }
    \hline 
    \textbf{Name} & \textbf{Objective function $\mathcal{F}(x)$}& \textbf{Search space} & \textbf{$x^\ast$} & \textbf{$\mathcal{F}(x^\ast)$} \\
    \hline \hline \rule{0pt}{16pt}
    Ackley & $-20\ \mbox{exp}\left( -0.2\sqrt{\frac{1}{\d}\sum_{i=1}^{\d}{(x_i)^2}}\right)-\mbox{exp}\left(\frac{1}{\d}\sum_{i=1}^\d{ \cos \left( 2\pi (x_i)\right) }\right) +20 +e$  & $[-32,32]^d$ & $ (0,\dots,0)$ & 0 
    \\ \hline  \rule{0pt}{16pt}
     Griewank & $1+\sum_{i=1}^{\d} \frac{(x_i)^2}{4000}-\prod_{i=1}^{\d} \mbox{cos}\left(\frac{x_i}{i}\right)$   & $[-600,600]^\d$ & $ (0,\dots,0)$ & 0 
    \\ \hline  \rule{0pt}{16pt}
        Rastrigin & $10\d + \sum_{i=1}^\d \left[(x_i)^2 - 10 \cos\left(2\pi (x_i) \right)\right]$  & $[-5.12,5.12]^\d$ & $ (0,\dots,0)$ & 0 
    \\ \hline  \rule{0pt}{16pt}
       Rosenbrock & $1-\cos{\left( 2\pi \sqrt{\sum_{i=1}^{\d}{(x_i)^2}}\right)} + 0.1 \sqrt{\sum_{i=1}^{\d}{(x_i)^2}} $  & $[-5,10]^\d$ & $ (1,\dots,1)$ & 0
    \\ \hline \rule{0pt}{16pt}
       Salomon & $1-\cos{\left( 2\pi \sqrt{\sum_{i=1}^{\d}{(x_i)^2}}\right)} + 0.1 \sqrt{\sum_{i=1}^{\d}{(x_i)^2}}  $  & $[-100,100]^\d$ & $ (0,\dots,0)$ & 0 
    \\ \hline \rule{0pt}{16pt}
     Schwefel 2.20 & $\sum_{i = 1}^\d \vert x_i \vert  $  & $[-100,100]^\d$ & $ (0,\dots,0)$ & 0 
    \\ \hline \rule{0pt}{16pt}
     XSY random & $\sum_{i=1}^{\d} \eta_i \vert x_i \vert^{i}, \qquad \eta_i \sim \textup{Unif}([0,1])$  & $[-5,5]^\d$ & $ (0,\dots,0)$ & 0
    \\ \hline \rule{0pt}{16pt}
     XSY 4 & $\left( \sum_{i=1}^{\d} \sin^2(x_i) - e^{ \ -\sum_{i=1}^{\d} (x_i)^2} \right) \, e^{ \ -\sum_{i=1}^{\d} \sin^2{\sqrt{\vert x_i \vert}}}   $  & $[-10,10]^\d$ & $ (0,\dots,0)$ & $-1$
    \\ \hline
  \end{tabular}}
 \end{center}
 \caption{Considered benchmark test functions for global optimization. For each function, the corresponding search space and global solution is given.}
 \label{tab:benchmarks}
\end{table}

As in plain CBO methods, we expect the most important parameters \gb{to be $\lambda$ and $\sigma$}, governing the balance between the exploitative behavior and the explorative one. In particular, we are interested in the algorithm performance as we change the ratio between $\lambda$ and $\sigma$. Therefore, in the first experiment we fix $\lambda = 0.01$, while considering different values of $\sigma$. The parameter $\alpha$ is adapted during the computation: starting form $\alpha_0 = 10$, it increases according to the law
\begin{align}
\label{alpha_adaptive}
\alpha = \alpha_0 \cdot k \cdot\log_2(k)\,.
\end{align}

Fig. \ref{fig:sigma_evaluation} shows the accuracy and the objective value reached for $\sigma \in [0,2]$ after $k_{\max} = 10^4$ algorithm iterations with $N = 200$ particles, \gb{and without} random selection. The optimal value for $\sigma$ is clearly problem-dependent, but we note that the optimal values for the problems considered all fall within a relative small range (underlined in gray in Fig. \ref{fig:sigma_evaluation}). 

\begin{figure}
\centering
\hfill
\begin{subfigure}{0.45\linewidth}
\includegraphics[width = 1\linewidth]{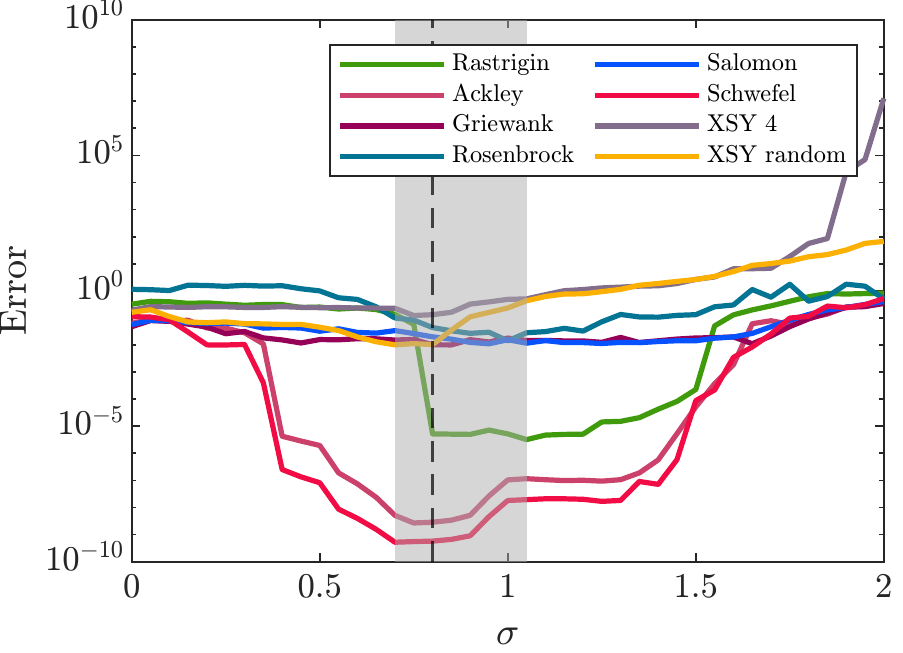}
\caption{$\|\bar y^{\alpha,k} -x^{\ast} \|_\infty $}
\end{subfigure}
\hfill
\begin{subfigure}{0.45\linewidth}
\includegraphics[width = 1\linewidth]{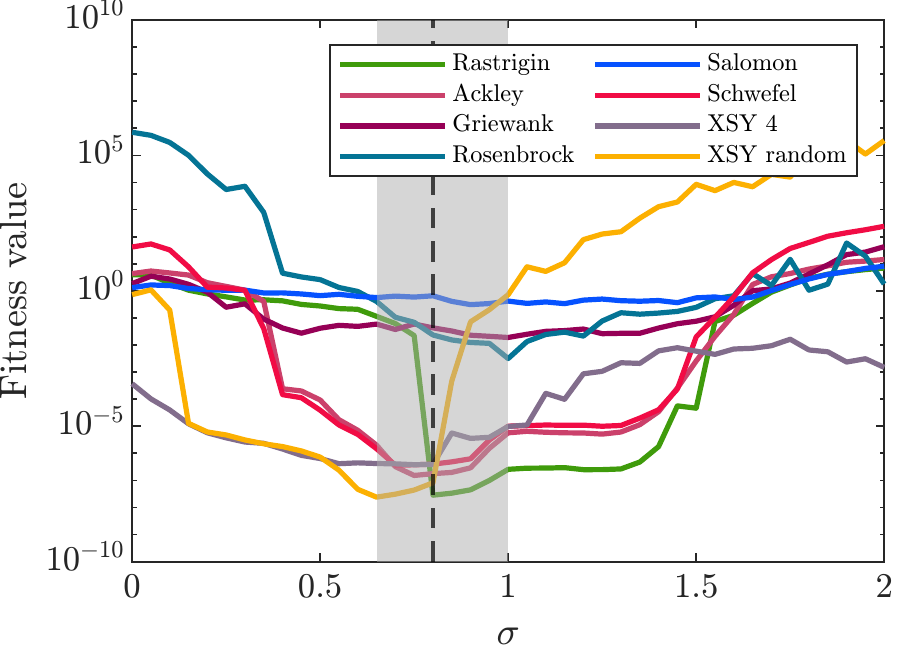}
\caption{$\F(\bar y^{\alpha,k})$}
\end{subfigure}
\hfill
\caption{Optimization on benchmark functions using CBO-ME. Behavior of the expectation error and fitness
value for different values of $\sigma$. Here $\lambda = 0.01$
and $\alpha$ is adaptive, with $\alpha_0 = 10$. The particle population is $N = 200$. Gray bands (of values $\left[0.70, 1.05 \right]$ for error and $\left[0.65, 1 \right]$ for fitness) show the range in which the minima of the different benchmark functions fall. The dotted line
marks the visually estimated pseudo-optimal value $\sigma = 0.8$.
\gb{Results are averaged on $250$ runs and obtained with  $k_{\max} = 10^4$ iterations,  without stopping criterion.}}
\label{fig:sigma_evaluation}
\end{figure}

%
%

From Fig. \ref{fig:sigma_evaluation} we infer that a good value for all benchmark problems considered is given by $\sigma =0.8$.
Using this value, we now compare CBO-ME, with plain CBO and the standard PSO (with and without alignment towards personal best) for different population sizes $N=50,100,200$. We keep the random selection mechanism off by setting $\mu=0$ and use the same previously chosen parameters when memory effects are used. 
For plain CBO, without memory effects, we set $\gb{\sigma =\sqrt{2}/2 \approx 0.71}$ \gb{and same $\lambda$}. Concerning PSO, we use the solver provided by the MATLAB  Global Optimisation Toolbox (\textmyfont{particleswarm}), changing the maximum number of iterations and the stall condition to the one used for CBO methods, to make the results comparable. The remaining parameters are kept as described in the relative documentation \cite{pedersen2010good}.
We set $k_\textup{max} = 10^4$, $\delta_{\textup{stall}} = 10^{-4}$ and consider a run successful when either
\be 
\|\bar y^{\alpha,k}-x^{\ast}\|_{\infty}< 0.1\quad \textup{or} \quad |\F(\bar y^{\alpha,k})-\F(x^{\ast})|< 0.01\,.
\label{eq:success}
\ee

Table \ref{tab:experiment1} reports success rate, final error given by $\| \bar y^{\alpha, k} - x^*\|_\infty$, \gb{objective function} value  \gb{$\F(\bar y^{\alpha,k})$} and total number of iterations, averaged over 250 runs. 
In addition to the classic PSO method, where the acceleration coefficients are chosen to be equal $C_1 =C_2 = 1.49$, Table \ref{tab:experiment1} also shows the results when only the alignment towards global best is considered in PSO ($C_1=0)$. 

While CBO already manages to find the global minimizer in most of the problems considered, we note that it \gb{sometimes} fails when Rastrigin, Rosenbrock or XSY random functions are optimized. \gb{CBO-ME outperforms CBO in these objective functions}. 
Standard PSO in many cases fails to solve the problem, see e.g.
 Rastrigin, Salomon or XSY 4 functions. PSO success rate is also lower among all problems, with the exception of the Schwefel 2.20 benchmark problem. 
Considering only global adjustment seems to show advantages 
except in the case of Ackley where setting $C_1=0$ decreases the success rate or, in the case of XSY 4, Salomon or Rastrigin, where convergence is not achieved even for $C_1 = 0$.
Consensus methods, however, perform better in terms of both success rate and speed up. In addition, for most problems, the population size $N$ seems not to play a significant role in the algorithms performance. This further motivates the introduction of the random selection strategy described in the Section \ref{sec:updaterule} in order to save computational costs.

\begin{table}
\begin{center}\resizebox{6.3in}{2.2in}{
\scriptsize{
\begin{tabular}{|l|lccc|ccc|ccc|ccc|}
\hline
 & \multicolumn{4}{c}{ \hspace{5pt} CBO ($\sigma = \sqrt{2}/2$)} & \multicolumn{3}{|c|}{ \hspace{5pt} CBO-ME ($\sigma = 0.8$)\hspace{5pt} } & \multicolumn{3}{|c|}{ \hspace{5pt} PSO \hspace{5pt} }& \multicolumn{3}{|c|}{ \hspace{5pt} PSO ($C_1 = 0$) \hspace{5pt} } \\
\hline
& & $N=50$ & $N=100$ & $N=200$ & $N=50$ & $N=100$ & $N=200$ & $N=50$ & $N=100$ & $N=200$ & $N=50$ & $N=100$ & $N=200$ \\  
\hline
{\rm \textbf{Ackley}}& Rate  &99.8\% & 100.0\% & 100.0\% & 100.0\% & 100.0\% & 100.0\% &17.1\%  & 41.2\% &  54.3\% &  4.2\% & 16.2\% & 40.1\% \\
& Error &  4.22e-06 & 2.14e-06 & 3.55e-06           & 2.42e-06 & 1.89e-06 & 1.56e-06 & 6.17e-09 &   8.86e-11 &  2.01e-12 & 2.23e-08 & 1.80e-10 & 8.65e-13\\
& $\mathcal{F}$ &1.18e-04 & 5.81e-05 & 7.30e-05 &1.54e-04 & 4.96e-05 & 4.99e-05 &   6.24e-09 &  7.65e-11 & 1.94e-12 & 2.06e-08 & 1.70e-10 & 8.01e-13\\
& Iterations &912.3 & 718.1 & 623.2			&977.2 & 703.3 & 622.2 & 501.8 & 424.2 & 341.3 & 502.3 & 421.2 & 321.2\\
\hline
{\rm \textbf{Griewank}}& Rate  &100.0\% & 100.0\% & 100.0\% 	&100.0\% & 100.0\% & 100.0\% &46.0\% &  48.6\%  &  55.3\% & 50.0\% & 58.7\% & 78.0\% \\
 & Error & 2.20e-02 & 2.21e-02 & 2.24e-02			      &  2.13e-02 & 2.16e-02 & 2.25e-02 &  7.34e-02 & 1.56e-02 & 9.45e-03 & 1.17e-01 & 1.10e-01 & 8.96e-02\\
& $\mathcal{F}$ &5.26e-02 & 5.31e-02 & 5.47e-02 	      &4.95e-02 & 5.15e-02 & 5.82e-02 &3.23e-03 & 4.11e-03 & 3.78-03 & 3.73e-03 & 3.71e-03 & 2.90e-03  \\
& Iterations &922.1 & 723.3 & 633.2				 &911.2 & 778.3 & 635.4  &512.2 & 403.1 &  399.2 & 432.1 & 391.2 & 311.2\\
\hline
{\rm \textbf{Rastrigin}}& Rate & 12.1\% & 34.3\% & 62.7\% &23.2\% & 69.7\% & 89.1\% &0.0\% & 0.0\% & 0.0\%  &0.0\% & 0.0\% & 0.0\%\\
& Error & 1.28e-04 & 1.83e-04 & 2.34e-04 &  9.73e-05 & 1.27e-04 & 1.76e-04 &  - & - & - &  - & - & -\\
& $\mathcal{F}$ & 4.51e-06 & 9.03e-06 & 1.46e-05       &2.54e-06 & 4.31e-06 & 8.28e-06 & - & - & - &  - & - & -\\
& Iterations &1083.0 & 933.7 & 819.8 &  1007.6 & 922.5 & 769.9  &10000.0 & 10000.0 &10000.0 & 10000.0 & 10000.0 &10000.0\\
\hline
{\rm \textbf{Rosenbrock}}& Rate  & 65.3\% & 86.7\% & 100.0\% 	&70.1\% & 94.2\% & 100.0\%  &9.3\% & 22.6\% & 36.6\% & 46.7\% & 60.7\% & 76.7\%\\
& Error &  1.84e-02 & 2.43e-02 & 1.42e-02 				& 3.60e-02 & 4.01e-02 & 1.82e-02 & 6.19e-04 & 2.56e-04 & 1.67e-04 & 4.44e-02 & 4.45e-02 & 4.46e-02\\
& $\mathcal{F}$ & 6.13e-03 & 7.57e-03 & 2.40e-03 		&1.26e-02 & 1.42e-02 & 2.65e-03 & 3.80e-02 & 3.76e-02 & 2.56e-02 & 2.56e-03 & 8.95e-04 & 3.71e-04\\
& Iterations&5773.2 & 5423.2 & 5233.1					&5933.2 & 4956.2 & 4155.2 &4822.2 & 3823.2 & 3026.3 & 5924.2 & 3834.1 & 2933.3\\
\hline
{\rm \textbf{Schwefel 2.20}}& Rate  & 100.0\% & 100.0\%       & 100.0\% &100.0\% & 100.0\% & 100.0\% &100.0\% & 100.0\% & 100.0\% & 100.0\% & 100.0\% & 100.0\% \\
& Error &  5.79e-06 & 8.23e-07 & 2.44e-07 				 &  8.42e-06 & 1.03e-06 & 2.76e-07 &  8.34e-10 & 1.97e-12 &  4.58e-14 & 1.68e-07 & 3.41e-10 & 8.03e-14\\
& $\mathcal{F}$ &1.04e-03 & 2.15e-04 & 8.36e-05 		&1.50e-03 & 3.12e-04 & 9.37e-05 & 1.94e-09 & 6.36e-12 & 1.52e-13 & 2.44e-07 & 6.48e-10 & 2.46e-13 \\
& Iterations &822.2 & 682.2 & 622.1 					&655.2 & 544.2 & 455.2 & 491.2 & 434.2 & 399.1 & 578.2 & 467.2 & 423.2\\
\hline
{\rm \textbf{Salomon}}& Rate  & 100.0\% & 100.0\% & 100.0\%		 &100.0\% & 100.0\% & 100.0\%  &0.0\% & 0.0\% & 0.0\% &0.0\% & 0.0\% & 0.0\%\\
& Error &3.12e-02 & 2.14e-02 & 1.87e-02 					& 5.28e-02 & 4.49e-02 & 3.91e-02 &  - & - & - &  - & - & -\\
& $\mathcal{F}$ &3.14e-01 & 2.15e-01 & 1.88e-01			&2.44e-01 & 1.86e-01 & 1.91e-01  & - & - & - &  - & - & - \\
& Iterations &10000.0 & 10000.0 & 10000.0 					&8872.2 & 9021.2 & 5356.5 & 10000.0 & 10000.0 & 10000.0 & 10000.0 & 10000.0 & 10000.0\\
\hline
{\rm \textbf{XSY random \ }}& Rate & 52.3\% & 81.7\% & 92.6\% &100.0\% & 100.0\% & 100.0\% &3.2\% & 17.1\% & 31.2\% &100.0\% & 100.0\% & 100.0\%\\
& Error &   2.64e-02 & 1.62e-02 & 9.80e-03				  & 3.06e-02 & 1.86e-02 & 1.15e-02 & 2.25e-01 & 9.56e-02 & 8.42e-02  & 6.23e-02 & 5.12e-02 & 2.34e-02\\
& $\mathcal{F}$ &6.95e-08 & 3.54e-08 & 2.13e-08		 & 2.21e-06 & 4.85e-08 & 3.17e-08 & 3.35e-04 & 2.28e-04 & 1.34e-04 & 8.22e-04 & 4.11e-04 & 3.45e-04\\
& Iterations &10000.0 & 10000.0 & 10000.0                              &10000.0 & 10000.0 & 10000.0 &10000.0 & 10000.0 & 10000.0 &10000.0 & 10000.0 & 10000.0\\
\hline
{\rm \textbf{XSY 4 \ }}& Rate &  27.2\% & 89.3\% &100.0\% 	 	 & 25.2\% & 91.2\% & 100.0\% & 0.0\% & 0.0\% & 0.0\% & 0.0\% & 0.0\% & 0.0\% \\
& Error & 8.10e-01 & 7.12e-01 & 7.89e-01 					& 8.01e-01 & 7.55e-01 & 6.17e-01 & - & - & - & - & - & -\\
& $\mathcal{F}$ &4.79e-07 & 3.78e-07 & 3.46e-07 			& 1.58e-06 & 8.56e-07 & 5.43e-07 & - & - & - & - & - & -\\
& Iterations &10000.0 & 10000.0 & 10000.0					&9733.2 & 9531.1 & 8733.2 &10000.0 & 10000.0 & 10000.0 &10000.0 & 10000.0 & 10000.0\\ 
\hline
\end{tabular}
}}
\end{center}
\caption{Comparison between classical CBO, CBO-ME and standard PSO with and without alignment towards personal best on benchmark problems. The solver \textmyfont{particleswarm} available in the MATLAB  Global Optimisation Toolbox was used for the results concerning the PSO method. Optimal choice of parameters, different for each method, are used for the CBO algorithms. Same stopping criterion and definition of success, see $\eqref{eq:success}$, were used. 
 Performance metric considered: success rate (see \eqref{eq:success}), error $\|\bar y^{\alpha, k}-x^*\|_\infty$, fitness value $\F(\bar y^{\alpha,k})$ and number of iterations. Results are averaged over 250 runs.
}
\label{tab:experiment1}
\end{table}

In the third experiment, we test the proposed random selection mechanism \eqref{eq:reduction} for different values of the parameter $\mu$. We recall that with $\mu=0$ we have no particles removal, while as $\mu$ increases, more particles are likely to be discarded when the system variance decreases. The initial population is set to $N_0 = 200$, while the minimum number of particles to $N_{\min} =10$. Results are reported in Tables \ref{tab:reduction1} and \ref{tab:reduction2} in terms of: success rate, error, objective value, weighted number of iterations, given by 
\be
w_{\textup{iter}} =  \sum_{k=0}^{k_{\textup{end}}} \frac{N_k}{N_0}\,,
\label{eq:witer}
\ee
and percentage of Computational Time Saved ($CTS$). Results show that relative large values of $\mu$ allow to reach fast convergence without affecting the algorithm performance.
In our experiments, the Rastrigin problem allows for larger values of $\mu$, while the Rosenbrock one seems to be more sensitive to the selection mechanism with respect to the other objectives. This justifies the different values of $\mu$ considered in Tables \ref{tab:reduction1} and \ref{tab:reduction2}.
 In both cases, a suitable value of $\mu$ reduces the computational time with almost no impact in terms of accuracy.

\begin{table}
\begin{center}
\scriptsize{
\begin{tabular}{|l|lc|c|c|c|}
\hline
& & $\mu = 0$ & $\mu = 0.05$ & $\mu = 0.1$ & $\mu = 0.2$ \\  
\hline
{\rm \textbf{Ackley}}& Rate & 100.0\% & 100.0\% & 100.0\%\ & 100.0\%\\
& Error & 1.89e-06 & 2.78e-06 & 6.12e-06 & 2.29e-05\\
& $\F$  & 9.21e-05 & 5.77e-05 & 2.12e-04 & 5.12e-04\\
& $w_{\textup{iter}}$ & 688.2 & 502.3 & 387.2 & 178.2\\
& $CTS$ & - &  31.3\% & 52.1 \% & 70.2\% \\
\hline
{\rm \textbf{Griewank}}& Rate & 100.0\% & 100.0\% & 100.0\%\ & 100.0\%\\
& Error & 2.12e-02 & 2.13e-02 & 2.18e-02 & 2.21e-02\\
& $\F$  & 5.80e-02 & 5.12e-02 & 5.90e-02 & 5.21e-02\\
& $w_{\textup{iter}}$ & 634.2 & 400.1 & 202.3 & 191.2\\
& $CTS$ & - & 31.3\% & 58.0\% & 71.6\%\\
\hline
{\rm \textbf{Schwefel 2.20}}& Rate & 100.0\% & 100.0\% & 100.0\%\ & 100.0\%\\
& Error & 2.16e-07 & 8.89e-07 & 8.12e-07 & 2.34e-08\\
& $\F$  & 9.11e-05 & 3.02e-05 & 1.23e-05 & 3.22e-05\\
& $w_{\textup{iter}}$ & 465.2 & 360.1 & 320.2 & 191.1\\
& $CTS$ & - & 24.7\% & 33.2\% & 62.1\%\\
\hline
{\rm \textbf{Salomon}}& Rate & 100.0\% & 100.0\% & 100.0\%\ & 100.0\%\\
& Error & 4.13e-02 & 3.37e-02 & 2.77e-02 & 1.69e-02\\
& $\F$  & 4.21e-01 & 4.22e-01 & 4.10e-01 & 3.67e-01\\
& $w_{\textup{iter}}$ &  2455.1 & 1551.1 & 1242.3 & 892.3\\
& $CTS$ & - & 38.2\% & 50.2\% & 66.2\%\\
\hline
{\rm \textbf{XSY random \ }}& Rate & 100.0\% & 100.0\% & 100.0\%\ & 100.0\%\\
& Error & 1.54e-02 & 8.34e-02 & 8.90e-02 & 9.23e-02\\
& $\F$  & 6.34e-07 & 2.05e-05 & 6.34e-05 & 2.33e-04\\
& $w_{\textup{iter}}$ & 10000.0 & 2821.3 & 1921.7 & 1167.2\\
& $CTS$ & - &  70.2\% & 85.3\% & 89.7\%\\
\hline
{\rm \textbf{XSY 4 \ }}& Rate & 100.0\% & 100.0\% & 100.0\%\ & 100.0\%\\
& Error & 5.37e-01 & 3.90e-01 & 1.55e-01 & 1.67e-01\\
& $\F$  & 1.19e-05 & 6.23e-06 & 3.67e-06 & 3.99e-06\\
& $w_{\textup{iter}}$ & 8945.1 & 3967.3 & 1923.4 & 1055.7\\
& $CTS$ & - &  50.2\% & 69.3\% & 85.6\%\\
\hline
\end{tabular}}
\end{center}
\caption{CBO-ME algorithm with random selection of particles tested against different benchmark functions with different values of $\mu$, which regulates the random selection mechanism. The system is initialized with $N_0 = 200$ particles and $\sigma = 0.8$. Performance metric considered: success rate (see \eqref{eq:success}), error $\|\bar y^{\alpha, k}-x^*\|_\infty$, fitness value $\F(\bar y^{\alpha,k})$, weighted iteration \eqref{eq:witer}, and Computational Time Saved (CTS). Results are averaged over 250 runs.
}
\label{tab:reduction1}
\end{table}

\begin{table}
\begin{center}\scriptsize{
\begin{tabular}{|l|lc|c|c|c|}
\hline
& & $\mu = 0$ & $\mu = 0.1$ & $\mu = 0.2$ & $\mu = 0.5$ \\  
\hline
{\rm \textbf{Rastrigin}}& Rate & 100.0\% & 100.0\% & 100.0\%\ & 100.0\%\\
& Error & 9.22e-05 & 7.76e-05 & 3.54e-05 & 1.34e-05\\
& $\F$  & 2.90e-06 & 2.99e-06 & 1.45e-06 & 1.12e-06\\
& $w_{\textup{iter}}$ & 1150.3 & 720.6 & 250.5 & 106.3\\
& $CTS$ & - &  39.2\% & 78.9\% & 92.3\%\\
\hline
& & $\mu = 0$ & $\mu = 0.01$ & $\mu = 0.02$ & $\mu = 0.05$ \\  
\hline
{\rm \textbf{Rosenbrock}}& Rate & 100.0\% & 100.0\% & 99.4\%\ & 99.0\%\\
& Error & 2.12e-02 & 2.21e-02 & 1.78e-02 & 1.45e-02\\
& $\F$  & 4.22e-03 & 5.67e-03 & 4.12e-03 & 4.45e-03\\
& $w_{\textup{iter}}$ & 3189.3 & 840.3 & 350.3 & 102.3\\
& $CTS$ & - &  75.3\% & 90.2\% & 92.4\%\\
\hline
\end{tabular}}
\end{center}
\caption{CBO-ME algorithm with particle reduction tested against Rastrigin and Rosenbrock functions with an higher diffusion parameter $\sigma = 1.1$ and for different values of $\mu$, which regulates the random selection mechanism. The system is initialized with $N_0 = 200$ particles. Performance metric considered: success rate (see \eqref{eq:success}), error ($\|\bar y^{\alpha, k}-x^*\|_\infty$), fitness value $\F(\bar y^{\alpha,k})$, weighted iteration \eqref{eq:witer}, and Computational Time Saved (CTS).}
\label{tab:reduction2}
\end{table}


Fig.s \ref{fig:ackley} and \ref{fig:rastrigin} show error and fitness value as a function of the number of fitness evaluations during the algorithm computation for Ackley and Rastrigin problems, respectively. Several values of $\mu$ are considered to display how the random selection mechanism affects the convergence speed. Initial particle population is set to $N_0 = 10^4$ and particles evolve for $k_{\textup{max}} = 10^4$ iterations. We note how convergence speed increases as $\mu$ increases.



\begin{figure}
\centering
\hfill
\begin{subfigure}{0.43\linewidth}
\includegraphics[width = 1\linewidth]{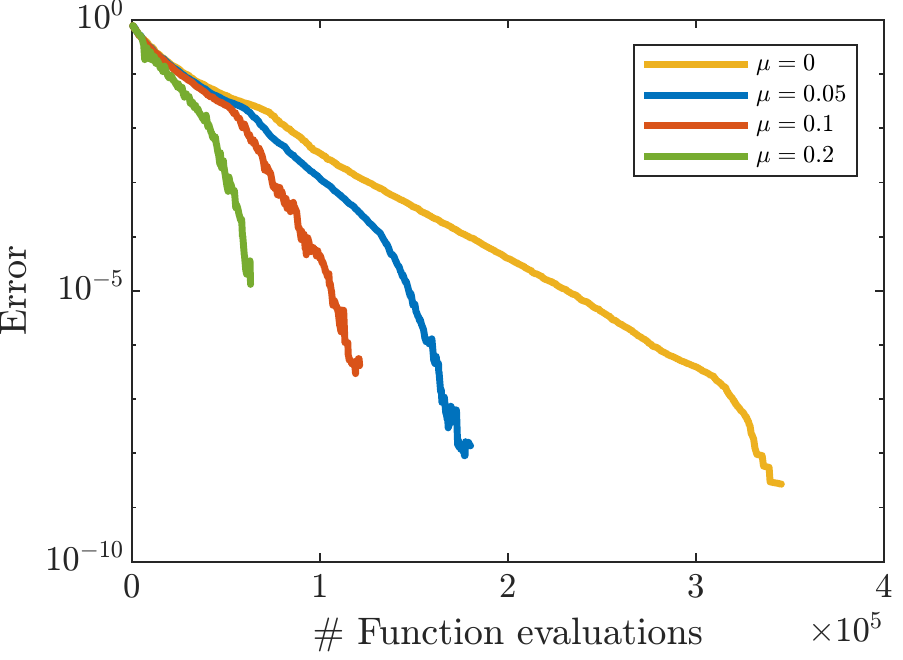}
\caption{Error: $\|\bar y^{\alpha,k} -x^{\ast} \|_\infty $} 
\end{subfigure}
\hfill
\begin{subfigure}{0.43\linewidth}
\includegraphics[width = \linewidth]{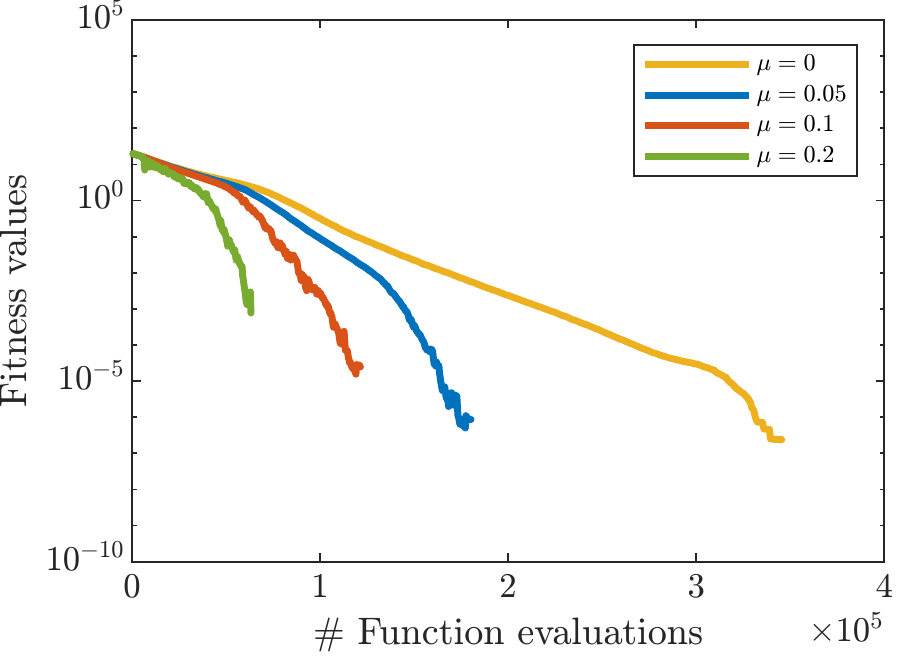}
\caption{Fitness value: $\F(\bar y^{\alpha,k})$}
\end{subfigure} 
\caption{Optimization of Ackley function for different values of the random selection parameter $\mu$, where the initial particle population is $N_0 = 10^4$. We report error (on the left) and fitness values (on the right) as the number of function evaluations increases.  Parameters are set as $\lambda = 0.01$, $\sigma = 0.8$, 
$\alpha$ adaptive starting from $\alpha_0 = 10$ and following the law $\alpha = \alpha_0 \cdot k \cdot\log_2(k)$. Results are averaged over 250 runs.
}
\label{fig:ackley}
\end{figure}


\begin{figure}[H]
\centering
\hfill
\begin{subfigure}{0.43\linewidth}
\includegraphics[width = 1\linewidth]{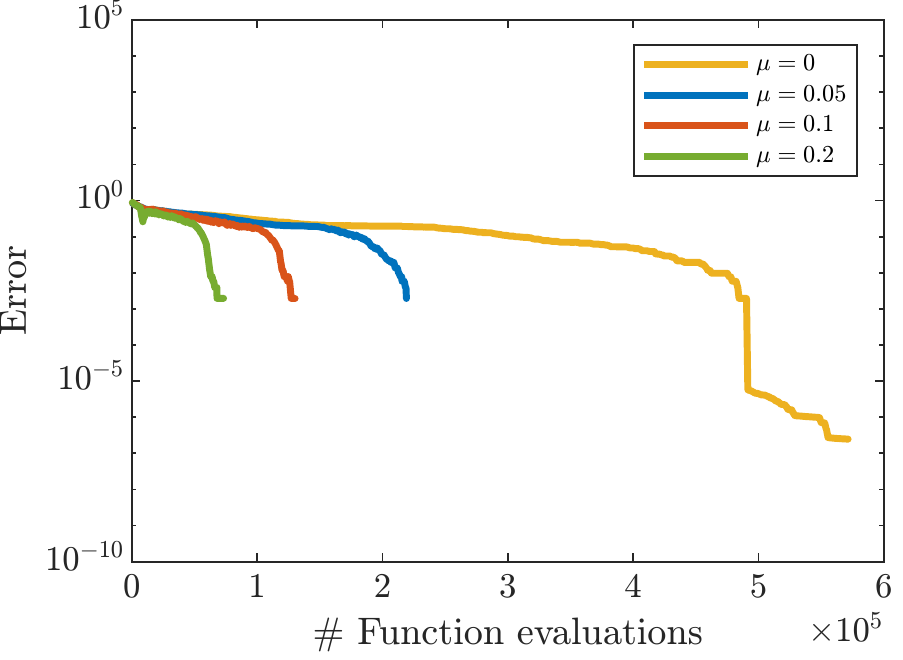}
\caption{Error: $\|\bar y^{\alpha,k} -x^{\ast} \|_\infty $}
\end{subfigure}
\hfill
\begin{subfigure}{0.43\linewidth}
\includegraphics[width = 1\linewidth]{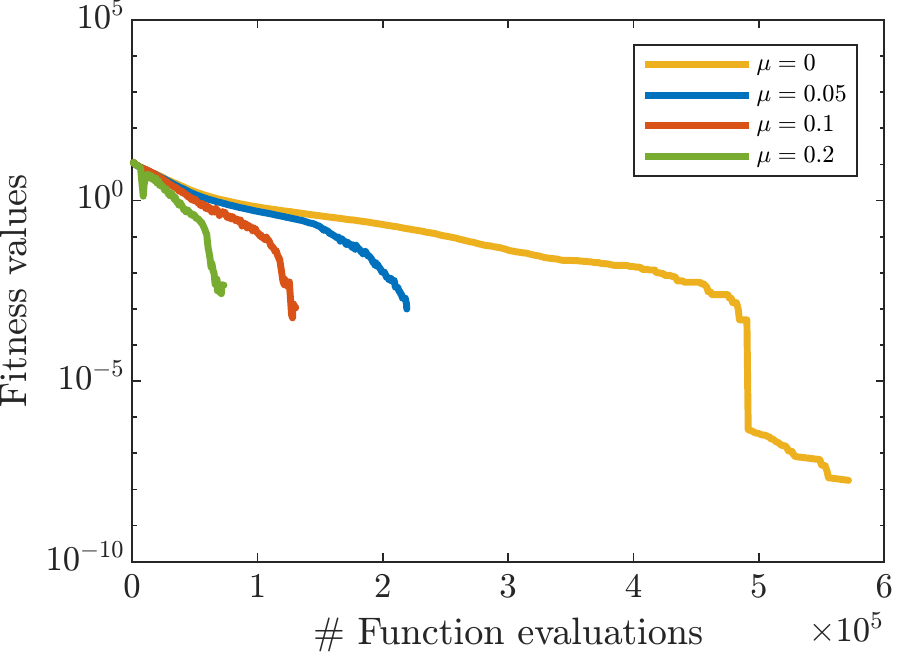}
\caption{Fitness value: $\F(\bar y^{\alpha,k})$}
\end{subfigure}
\hfill
\caption{Optimization of Rastrigin function for different values of the random selection parameter $\mu$ where the initial particle population is $N_0 = 10^4$. We report error (on the left) and fitness values (on the right) as the number of function evaluations increases. Parameters are set as $\lambda = 0.01$, $\sigma = 1.1$, 
$\alpha$ adaptive starting from $\alpha_0 = 10$ and following the law $\alpha = \alpha_0 \cdot k \cdot\log_2(k)$. Results are averaged over 250 runs. 
 }
 \label{fig:rastrigin}
\end{figure}


\subsection{Applications}
In this section, we propose some applications of the proposed optimization algorithm. First we consider a image segmentation problem using multi-thresholding, then we use the CBO-ME to train a Neural Network (NN) architecture to approximate functions and perform image classification on the MNIST database of handwritten digits. 

\subsubsection{Image segmentation}

To perform image segmentation, we use a threshold detection technique, namely, the multidimensional Otsu algorithm \cite{otsu1979threshold,tsai2012novel} in order to compare the results to similar optimization algorithm, such as the Modified PSO in \cite{tian2018mpso}.

\gb{
In the Otsu algorithm, every pixel of the image is assigned to one of the possible $L$ grayscale values. We denote with $\eta_i$ the number of pixel with gray level $i$, $1 \leq i \leq L$ and $N_{pix}= \sum_{i=1}^L \eta_i$ the total number of pixels  \cite{otsu1979threshold}.  We consider an extension of Otsu's technique to the multidimensional case \cite{tsai2012novel} to test capabilities of method.
Assuming we want to optimize the choice of $\d$ thresholds, we require $\d+1$ classes of different gray-scales $(C_0, \dots, C_\d)$ with relative probabilities of occurrence classes defined as
\begin{align*} 
\omega_0(l_1) = \sum_{i = 1}^{l_1} p_i \ , \quad \dots , \quad \ \omega_\d(l_\d) = \sum_{i = l_\d+1}^{L} p_i, \quad p_i = \frac{\eta_i}{N_{pix}}
\end{align*}
and classes mean levels
\begin{align*}
\mu_0(l_1) = \frac{\sum_{i = 1}^{l_1} i p_i}{\omega_0} , \quad  \dots ,\quad \mu_\d(l_\d) = \frac{\sum_{i = l_{\d+1}}^{L} i p_i}{\omega_\d}, 
\end{align*}
The optimal thresholds $(\hat l_1,\dots,\hat l_\d )$ are those that satisfy $\hat l_1<\dots<\hat l_\d$ and maximise
\begin{align} \label{eq: funmin2}
f(l_1, \dots l_{\d}) = \sum_{i=1}^{\d} \omega_i(l_i) \mu_i^2(l_i)\,.
\end{align}
}

For the experiment, we chose $\d = 5 $ thresholds and compare the segmentation performed by Otsu's method, solved with both standard PSO and CBO-ME, with segmentation obtained by dividing the grayscale into $\d+1$ uniformly spaced intervals. For PSO, we use to the default parameters in the \texttt{particleswarm} function in the MATLAB Global Optimisation Toolbox, while for CBO-ME we used optimal parameters found in Section \ref{sec:benchmark} and exploit the random selection technique to speed up the algorithm.

We report the results on two sample images, Fig.s \ref{fig:segm_woman} and \ref{fig:segm_lake}. We fix $k_{\max} = 10^3$ and average results over $250$ runs. 
As in \cite{arora2008multilevel}, we evaluate multi-thresholding segmentation through the Peak Signal to Noise Ratio ($PSNR$) computed as:
\begin{align*} 
PSNR = 20 \cdot \log_{10} \left( \frac{255}{RMSE} \right)
\end{align*}
where $RMSE$ is the Root Mean-Squared Error, defined as
\begin{align*} 
RMSE = \sqrt{\frac{1}{N_{pix}}\sum^{N_{row}}_{i=1} \sum^{N_{col}}_{j=1} \left[I(i,j)-S(i,j)\right]^2}
\end{align*}
where $N_{pix} = N_{row}\cdot N_{col}$, $I$ is the original image and $S$ is the associated segmented image. The higher the value of $PSNR$ is, the greater the similarity between the clustered image and the original image is.
From Fig.s \ref{fig:segm_woman},\ref{fig:segm_lake}, we note that the most accurate segmentation on details is obtained by the CBO-ME method. This is quantitatively confirmed by the $PSNR$ values reported in Table \ref{tab:accuracy}. 


\begin{figure}
\centering
\begin{subfigure}{0.24\linewidth}
\includegraphics[width = 0.9\linewidth]{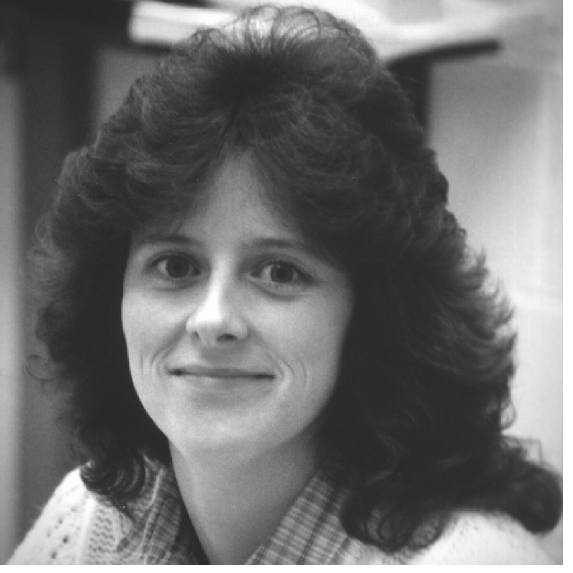} 
\caption{Original}
\end{subfigure}
\begin{subfigure}{0.24\linewidth}
\includegraphics[width = 0.9\linewidth]{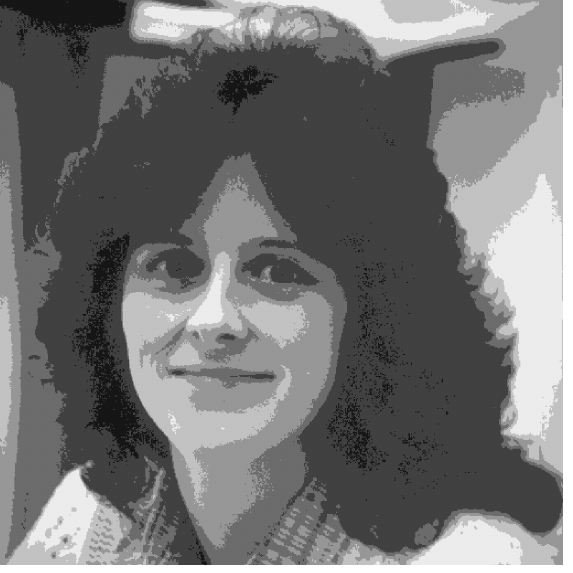} 
\caption{Standard segmentation}
\end{subfigure}
\begin{subfigure}{0.24\linewidth}
\includegraphics[width = 0.9\linewidth]{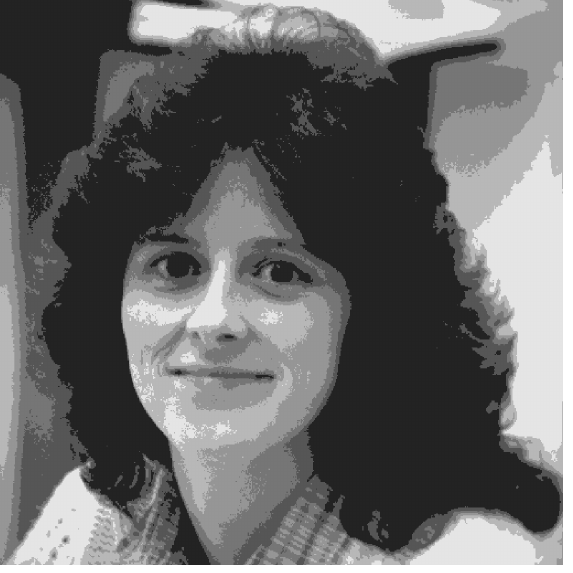} 
\caption{Otsu seg. (PSO)}
\end{subfigure}
\begin{subfigure}{0.24\linewidth}
\includegraphics[width = 0.9\linewidth]{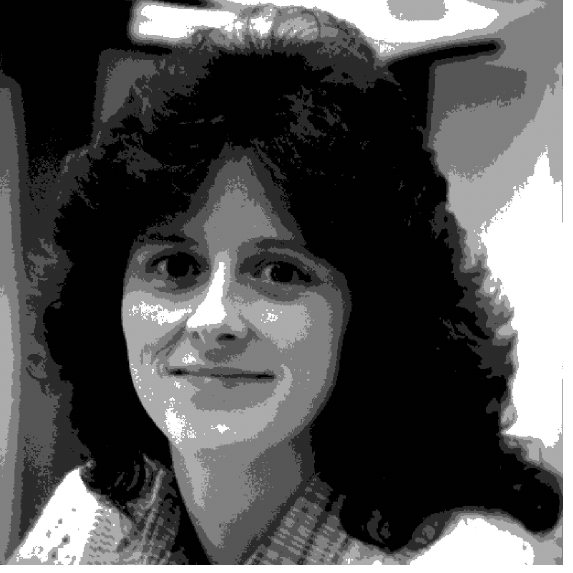} 
\caption{Otsu seg. (CBO-ME)}
\end{subfigure}
\caption{Image segmentation of \textmyfont{darkhair woman} image $(256 \times 256$ pixels$)$ with standard segmentation and Otsu segmentation solved respectively by PSO (c) and by CBO-ME (d). Results are averaged over 250 runs, with an initial population of $N_0 =10^3$ particles. 
 }
 \label{fig:segm_woman}
\end{figure}

\begin{figure}
\centering
\begin{subfigure}{0.24\linewidth}
\includegraphics[width = 0.9\linewidth]{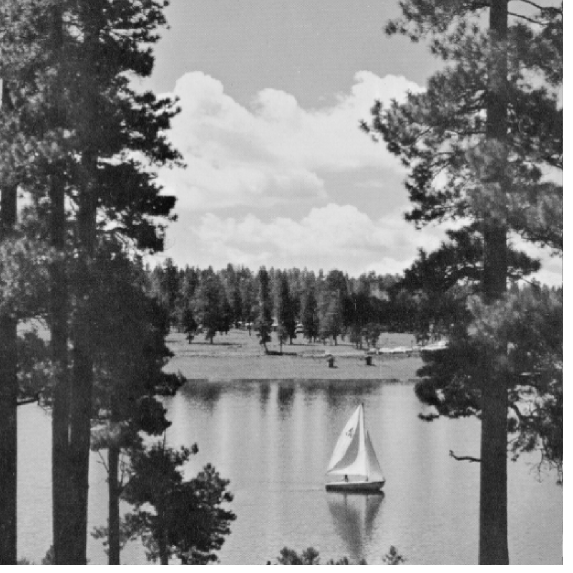} 
\caption{Original}
\end{subfigure} \hfill
\begin{subfigure}{0.24\linewidth}
\includegraphics[width = 0.9\linewidth]{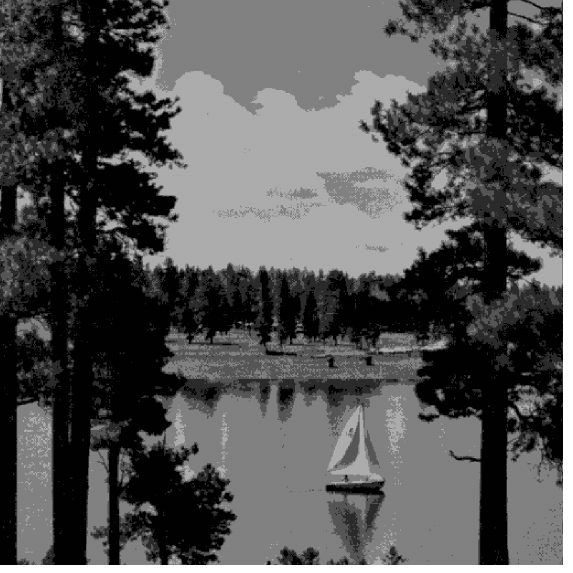} 
\caption{Standard segmentation}
\end{subfigure}\hfill
\begin{subfigure}{0.24\linewidth}
\includegraphics[width = 0.9\linewidth]{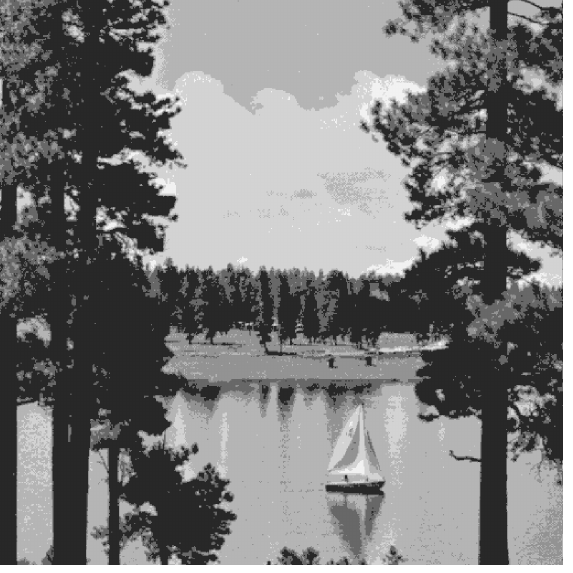} 
\caption{Otsu seg. (PSO)}
\end{subfigure} \hfill
\begin{subfigure}{0.24\linewidth}
\includegraphics[width = 0.9\linewidth]{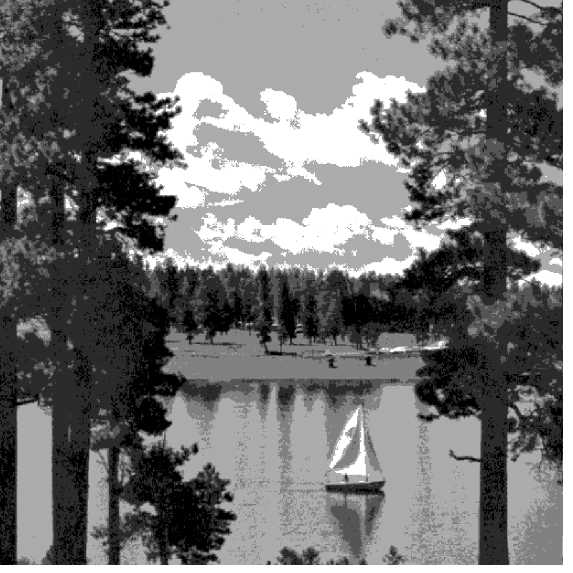} 
\caption{Otsu seg. (CBO-ME)}
\end{subfigure}
\caption{Image segmentation of \textmyfont{lake} image $(256 \times 256$ pixels$)$ with standard segmentation and Otsu segmentation solved respectively by PSO (c) and by CBO-ME (d). Results are averaged over 250 runs, with an initial population of $N_0 =10^3$ particles. 
}
\label{fig:segm_lake}
\end{figure}

\begin{table}
\begin{center}{
\scriptsize{ \renewcommand\arraystretch{1.3}{
\begin{tabular}{|l|c|c|c|c|c|}
\hline
&  \textmyfont{cameraman} & \textmyfont{lake} &  \textmyfont{lena} &  \textmyfont{peppers} &  \textmyfont{woman darkhair}  \\  
\hline  \hline
{ \rm \textbf{Standard segmentation}} &   22.83 & 21.72 & 24.35 & 27.24 & 25.33\\
\hline
{\rm \ \splitcell[b]{\textbf{Otsu segmentation} \\ \hspace{-20pt} \textbf{(PSO)}}} &  34.62 & 32.33 & 38.19 &  38.03 &  37.14 \ \\
\hline
{\rm \splitcell[b]{\textbf{Otsu segmentation} \\ \ \textbf{(CBO-ME)}}} &  {37.22} & {35.44} & {38.72} & {38.28} & {39.57}\ \\

\hline 
\end{tabular}}}}
\end{center}
\caption{$PSNR$ values obtained for $5$ sample images known in literature. We compared the Otsu segmentation solved by the proposed CBO-ME method with the classical PSO method and with equispaced thresholding segmentation. Experiments are performed with $\textup{d}=5$ thresholds.
}
\label{tab:accuracy}
\end{table}


\subsubsection{Approximating functions with NN}

In this section, we use the proposed CBO-ME algorithm to train a NN architecture into approximating a function $u: I \to \RR, I \subset \RR$ with low regularity.
As in \cite{jin2022adaptive}, we use a fully-connected NN with $m$ layers 
\be
f(x; \theta) =  ( L_m \circ \dots L_2 \circ L_1 )(x)
\label{eq:NN}
\ee
where each layer is given by
\be \notag
L_i = \sigma (W^i x + b^i)
\ee
with $\sigma(x) = 1/(1+ \exp(-x))$ being the component-wise sigmoid function. We use internal layers of dimension $n$, so $W^1 \in \RR^{n \times 1}, b^1 \in \RR$, $W^m \in \RR^{1\times n}, b^m \in \RR^\textup{d}$ and $W^i \in \RR^{n\times n}$ for all $i=2, \dots, m-1$. In \eqref{eq:NN}, all DNN parameters are collected in $\theta = \{W^i, b^i \}_{i=1}^m.$

As loss function which need to be minimized, we consider the $L^2$-norm between the target function $u$ and its NN approximation $f(\cdot \,; \theta)$
\be
\F( \theta): = \| f(\cdot\,; \theta) - u \|_{L^2(I)}\,. 
\ee
Again, similarly to \cite{jin2022adaptive}, we test the method against the following two functions:
\begin{align}
 \label{approx_smooth}
u_1(x) &= \sin(2\pi x) + \sin(8\pi x^2) \\
 \label{approx_nosmooth}
u_2(x) &= 
\begin{cases} 1  &\textup{if} \;\; x < -\frac{7}{8}, -\frac{1}{8}<x< \frac{1}{8}, x > \frac{7}{8}\,, \\
-1 &\textup{if} \;\; \frac{3}{8}<x< \frac{5}{8}, -\frac{5}{8}<x< -\frac{3}{8}, \\
0    &\textup{otherwise} \,.
\end{cases}
\end{align} 
We note that $u_1$ is smooth, while $u_2$ is discontinuous. Parameters of the CBO-ME algorithm have been set to $\lambda = 0.01, \sigma = 0.8$, 
as in the previous sections. Parameter $\alpha$ is adapted during the computation as in \eqref{alpha_adaptive} and random selection mechanism is used. We employ $m=3$ layers with internal dimension $n=50$. Results are displayed in Fig.s \ref{fig:u1} and \ref{fig:u2}. We note that convergence is slower for the discontinuous step function $u_2$. 

\begin{figure}
\centering
\begin{subfigure}{0.24\linewidth}
\includegraphics[width = 0.95\linewidth]{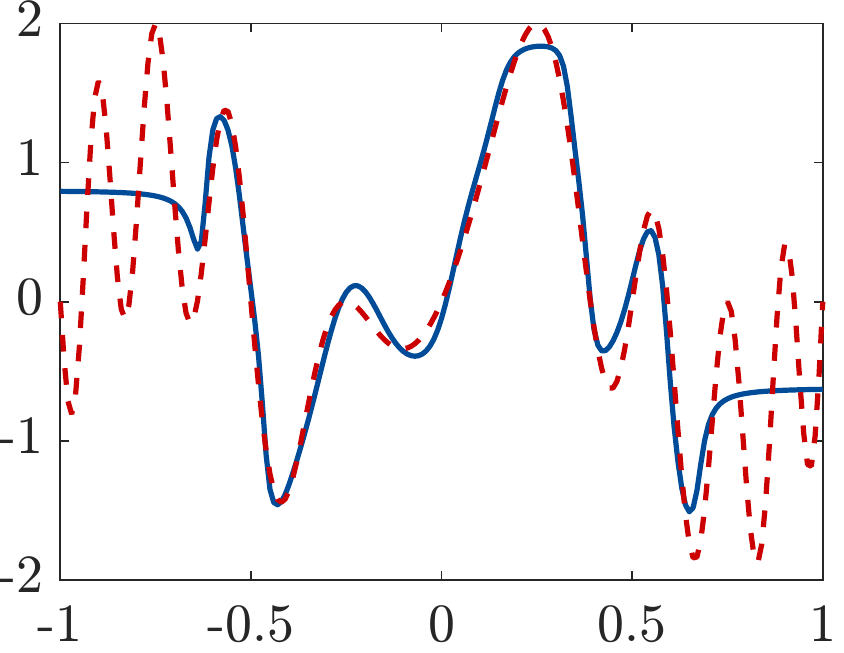} 
\caption{$2000$ epochs}
\end{subfigure}
\begin{subfigure}{0.24\linewidth}
\includegraphics[width = 0.95\linewidth]{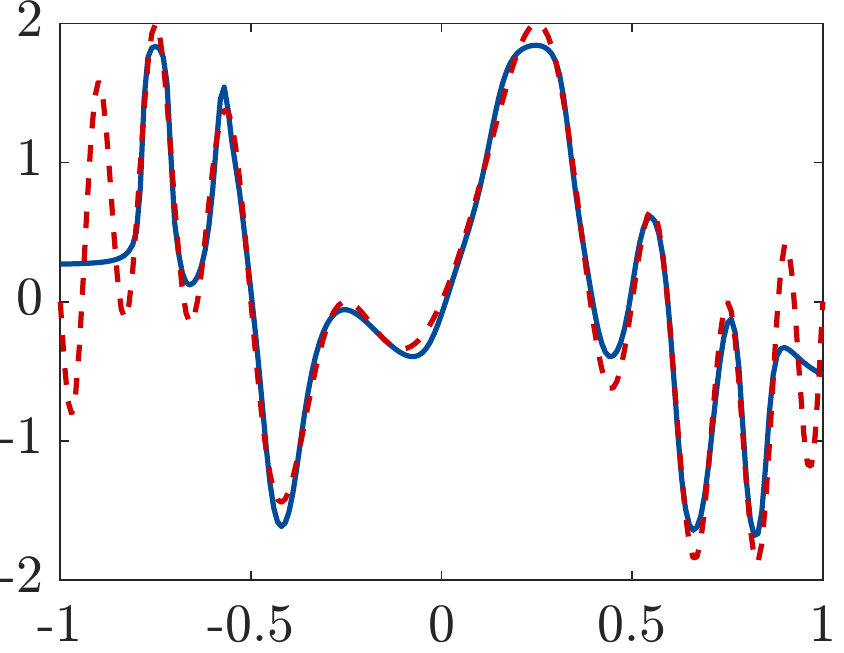} 
\caption{$3000$ epochs}
\end{subfigure}
\begin{subfigure}{0.24\linewidth}
\includegraphics[width = 0.95\linewidth]{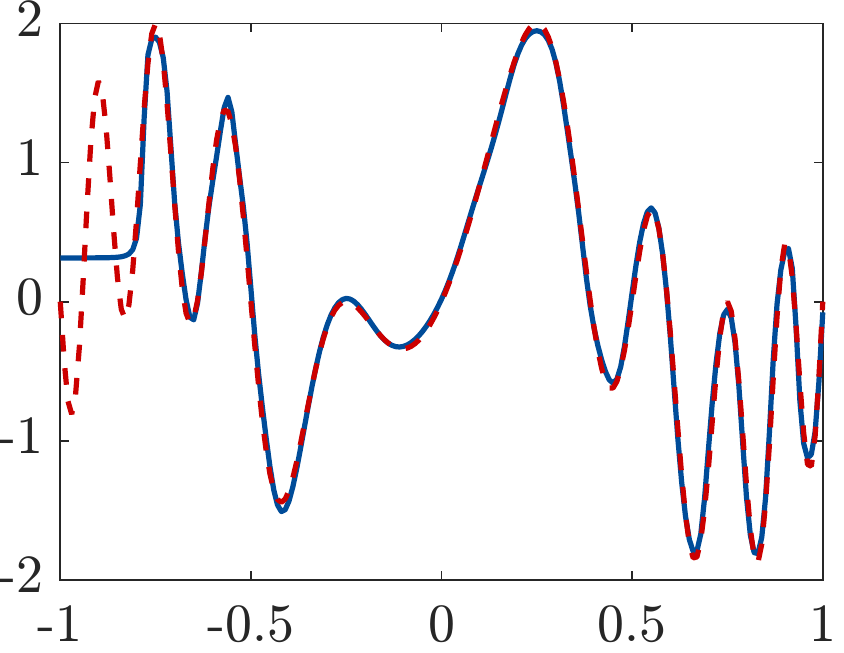} 
\caption{$5000$ epochs}
\end{subfigure}
\begin{subfigure}{0.24\linewidth}
\includegraphics[width = 0.95\linewidth]{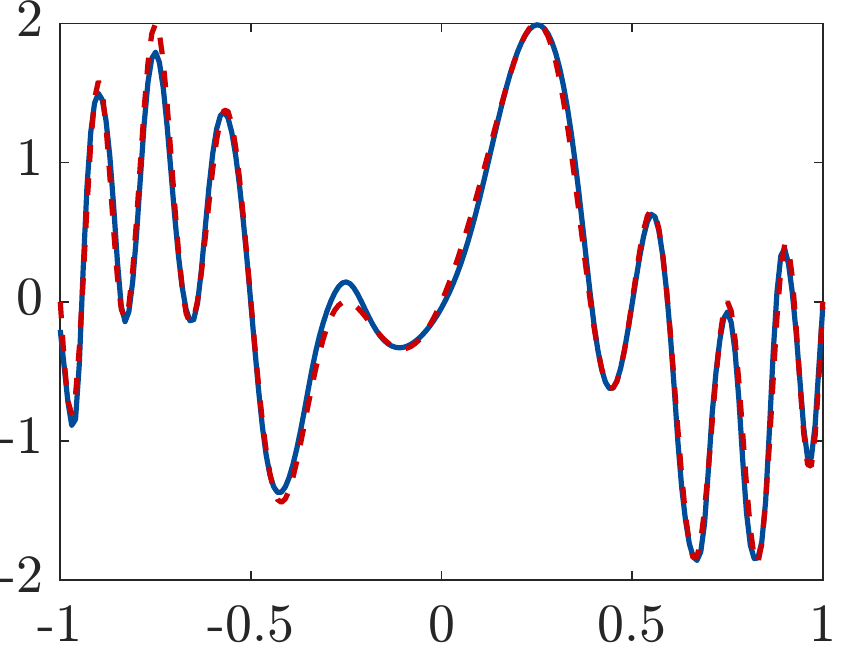} 
\caption{$8000$ epochs}
\end{subfigure}
\caption{Approximating smooth function $u_1$ \eqref{approx_smooth} using a network with $n = 50$ and  $m = 3$. We initially use $N_0 = 500$ particles and we set $\lambda = 0.01$, $\sigma = 0.8$ . Parameter $\alpha$ is adaptive, starting from $\alpha_0 = 10$.}
\label{fig:u1}
\end{figure}

\begin{figure}
\centering
\begin{subfigure}{0.24\linewidth}
\includegraphics[width = 0.95\linewidth]{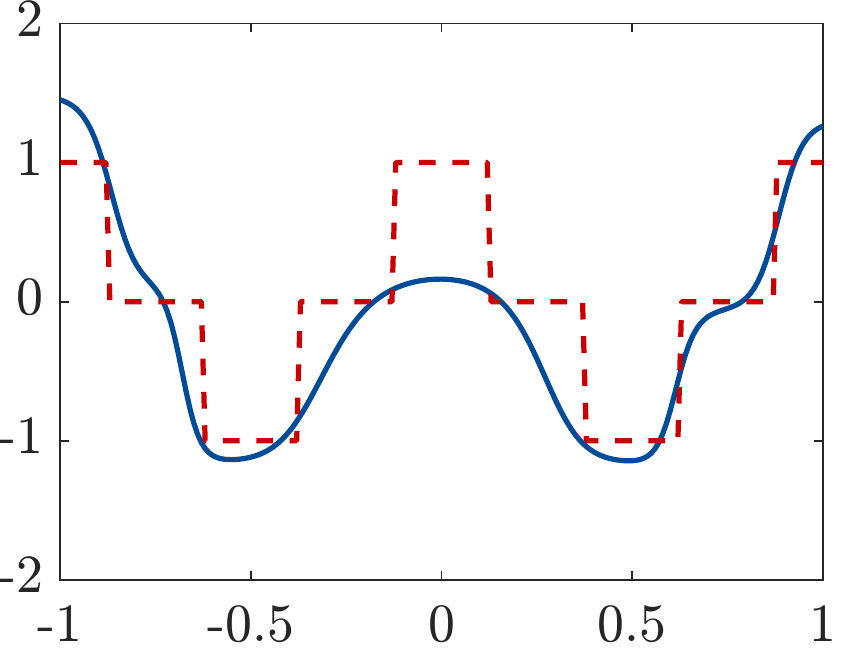} 
\caption{$2000$ epochs}
\end{subfigure}
\begin{subfigure}{0.24\linewidth}
\includegraphics[width = 0.95\linewidth]{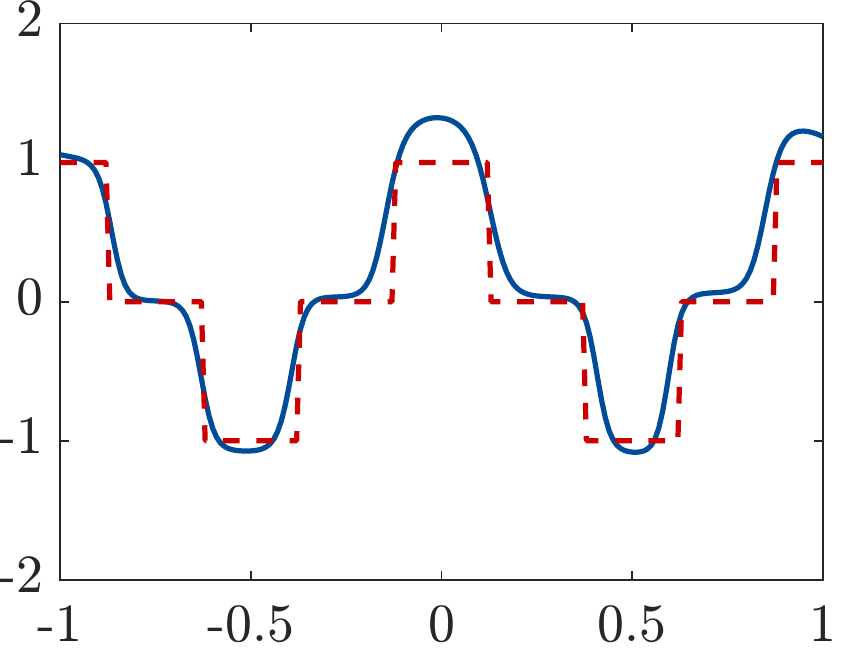} 
\caption{$3000$ epochs}
\end{subfigure}
\begin{subfigure}{0.24\linewidth}
\includegraphics[width = 0.95\linewidth]{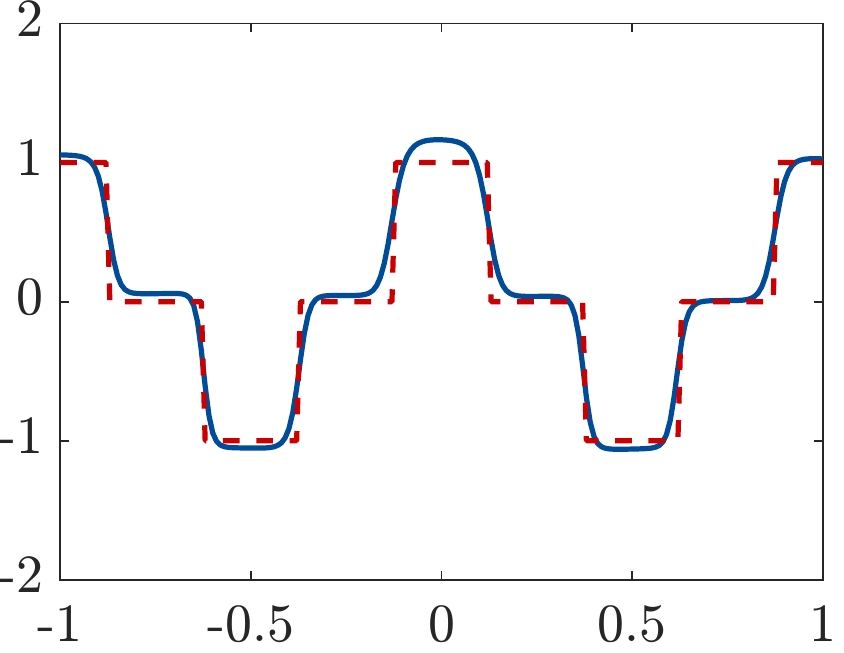} 
\caption{$5000$ epochs}
\end{subfigure}
\begin{subfigure}{0.24\linewidth}
\includegraphics[width = 0.95\linewidth]{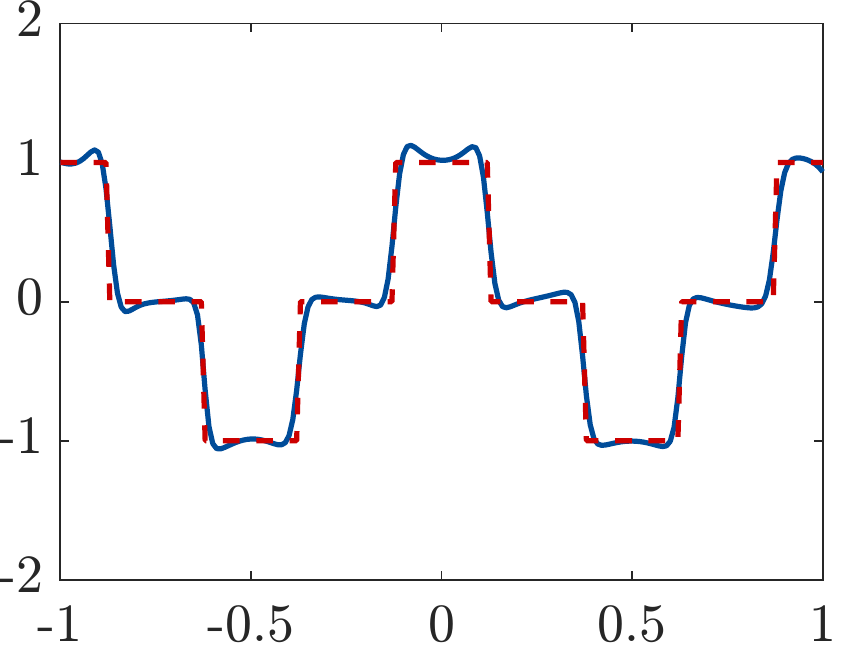} 
\caption{$8000$ epochs}
\end{subfigure}
\caption{Approximating non-smooth $u_2$ \eqref{approx_nosmooth} function using a network with $n = 50$, $m = 3$.We initially use $N_0 = 500$ particles and we set $\lambda = 0.01$, $\sigma = 0.8$ . Parameter $\alpha$ is adaptive, starting from $\alpha_0 = 10$.}
\label{fig:u2}
\end{figure}

%


\bigskip

\subsubsection{Application on MNIST dataset}

\gb{We now employ the proposed algorithm to train a NN architecture to solve a image classification task. We will consider the MNIST dataset \cite{mnist} composed of handwritten digits in gray scale with $28\times28$ pixels. For better comparability with CBO methods without memory effects, we follow the experiment settings used in the literature \cite{carrillo2019consensus,fornasier2022aniso,riedl2022leveraging,benfenati2021binary}, which we summarize below.

We consider a $1$-layer NN where input images $x \in \RR^{28\times28}$ are first vectorized  $x\mapsto \textup{vec}(x) \in \RR^{\gb{784}}$ and then processed through a fully-connected layer with parameters $\theta = \{W, b\}$, where $W \in \RR^{10\times \gb{784}}, b \in \RR^{10}$. That is, the network is given by 
\be
f^{\textup{SNN}}(x; \theta) = \textup{softmax}\left(\textup{ReLU}\left(W \textup{vec}(x) + b \right)\right) ,
\label{eq:shallowNN}
\ee
where $\textup{ReLU}(z) = \max\{z,0\}$ (component-wise) and $\textup{softmax}(z) = (e^{z_1}, \dots,e^{z_n})/(\sum_i e^{z_i})$ are the commonly used activation functions. 
During the training, batch regularization is performed after $\textup{ReLU}$ is applied in order to speed up convergence. Given a training set $\{ (x^m, \ell^m)\}_{m=1}^M$, $x_m \in \RR^{28 \times 28}, \ell_m \in \{0,1\}^{10}$ made of $M$ image-label tuples, we train the model by minimizing the categorical cross-entropy loss 
\be
\F (\theta) = \frac1M \sum_{m=1}^M \left( -  \sum_{i=1}^{10} \ell_{i}^m \log (f_i(x^m, \theta))    \right)\,.
\label{eq:lossMNIST}
\ee 
The entire training set is made of 60,000 images, 6,000 per class, but we divide it in batches of size $M = 120$, and consider a different batch at each algorithm iteration to evaluate \eqref{eq:lossMNIST}.

The initial population of $N_0$ particles is sampled from the standard normal distribution $\mathcal{N}(0, \mathbf{I}_\d)$ and we employ the particle reduction strategy given by \eqref{eq:reduction} with $N_{\textup{min}} = 100$. Differently from previous experiments, though, we compute the new number of particles $N_{k+1}$ based on the variance of the personal bests $\{y_i^{k} \}_{i \in N_k}$, rather then considering the particle positions $\{x_i^k \}_{i \in N_k}$. This is because, in this application, the variance of the particle positions shows an oscillatory behavior (see Fig. \ref{fig:MNIST_muvary_var}).

Following the mini-batch approach suggested in \cite{carrillo2019consensus},
at each algorithm iteration we divide the particle population in mini-batches of size $n_N = 20$ (the last one being eventually smaller) and independently perform the update within the different mini-batches. Particles are re-ordered after each update step, so that the mini-batches always vary during the computation.

Fig. \ref{fig:MNIST_muvary}, shows the algorithm performance when $N_0 = 1000$  particles  are initially used and random selection is performed with different parameters $\mu$. We note that, in terms of accuracy and loss, the performance is comparable. The random selection strategy sensibly reduces the number of function evaluations needed, especially when the particle system has already formed consensus. The number of particles per iteration is displayed in Fig. \ref{fig:MNIST_muvary_var}, further showing how the computational complexity of an update step decays during the computation.  

We also compare the algorithm performance when different initial population sizes $N_0 = 500, 1000, 2000$ are considered, with same random selection strength $\mu=0.1$. In this case, the best balance between computational cost and accuracy is given by $N_0 = 1000$. We note in particular how starting with a larger population size of $2000$ particles leads to a marginal improvement of the algorithm performance, while requiring a much higher number of loss evaluations.

In the last experiment, we compare algorithms CBO-ME and CBO without memory effects. We also consider a third heuristic proposed in \cite{carrillo2019consensus} and further tested in \cite{fornasier2022aniso}. In this CBO variant, no memory is employed, but, whenever $\| \bar x^{\alpha,k+1} - \bar x^{\alpha,k}\|_{\infty} \leq \delta$, particles are
 randomly perturbed before the CBO iteration, by adding Gaussian noise:
\be
x_i^k \gets x_i^k + \sigma \tilde \theta^k_i \quad \textup{with} \quad \tilde \theta^k_i \sim \mathcal{N}(0, \mathbf{I}_\textup{d})\, .
\label{eq:perturbation}
\ee
In our experiment, we also consider the above variant with $\delta = 10^{-5}$ and $\bar x^{\alpha,k}$ computed among the whole particle system.

Fig. \ref{fig:MNIST_CBO} illustrates the performance of the different algorithms for three different choices of parameter $\alpha$: increasing, fixed to $\alpha=50$ and fixed to $\alpha = 5 \cdot 10^4$. The drift parameter is set to $\lambda = 0.1$ while we set $\sigma = \sqrt{0.1}$ for CBO-ME and CBO without random perturbations and $\sigma = \sqrt{0.04}$ for CBO when random noise is added. Initial populations of $N_0 = 1000$ with selection parameter $\mu=0.2$ are used when there is no additional noise, while we use $N_0 = 250$ and no particle selection when we perform random perturbations. This is motivated by the fact that the particle variance, whenever noise is added, shows an oscillatory behavior which is not compatible with the mechanism of random selection. 

We note that CBO-ME performs better when lower values of $\alpha$ are used, while the effect of memory is reduced for larger values of $\alpha$. A low value of sigma, together with random perturbations, typically slow down the convergence of the particle system.

\begin{remark}
\begin{itemize}
\item In CBO literature, an additional parameter $\Delta t$ is typically used to define step size $\lambda = \tilde \lambda \Delta t$ and the diffusion strength $\sigma = \tilde \sigma \sqrt{\Delta t}$, for some $\tilde \lambda, \tilde \sigma$. This is because the particles update rule is interpreted as a numerical scheme solving a time-continuous dynamics, as we will see in the next Section. We decided here to avoid using $\Delta t$ for better comparability with PSO algorithms. We note how choosing, for instance, $\lambda = 0.1, \sigma = \sqrt{0.1}$ is equivalent to the parameters choice $\tilde \lambda = 1, \tilde \sigma = 1$ with $\Delta t = 0.1$.

\item Experiments show how both CBO and CBO-ME converges towards a solution in less than an epoch. This is coherent with other population-based algorithms, such as Ensemble Kalman Filter \cite{alper2020klaman}. Moreover, we note how adding noise  during the computation sensibly reduces the convergence speed, see Fig. \ref{fig:MNIST_CBO}.
\end{itemize}
\end{remark}
}

\begin{figure}
\begin{subfigure}{1\linewidth}
\centering
\includegraphics[trim  = 3mm 0 7.3cm 5mm, clip, height = 5.5cm]{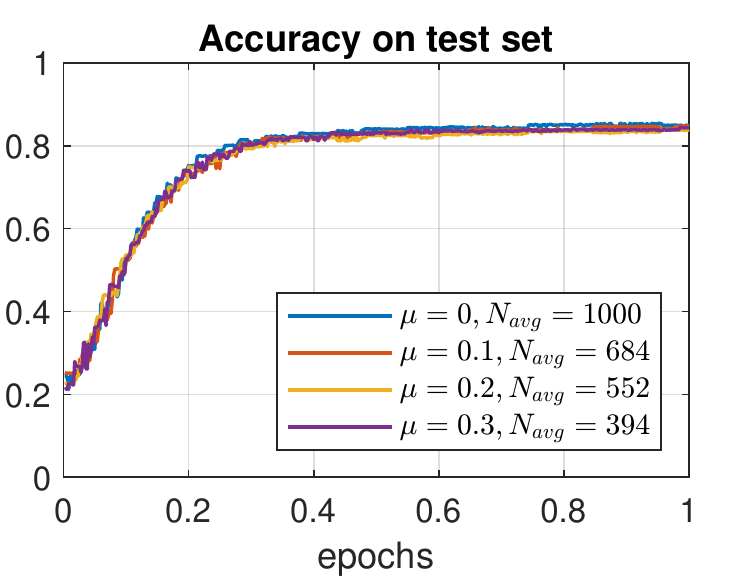} 
\includegraphics[trim  = 7cm 0 0 5mm, clip, height = 5.5cm]{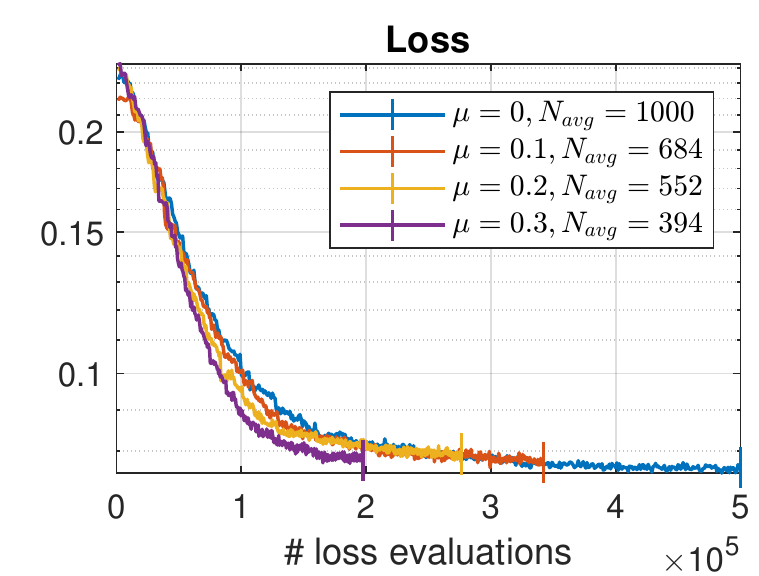} 
\caption{Different random selection parameters $\mu$, same initial population size $N_0 = 1000$}
\label{fig:MNIST_muvary}
\end{subfigure}\\
\begin{subfigure}{1\linewidth}
\centering
\includegraphics[trim  = 0mm 0 7.3cm 5mm, clip, height = 5.5cm]{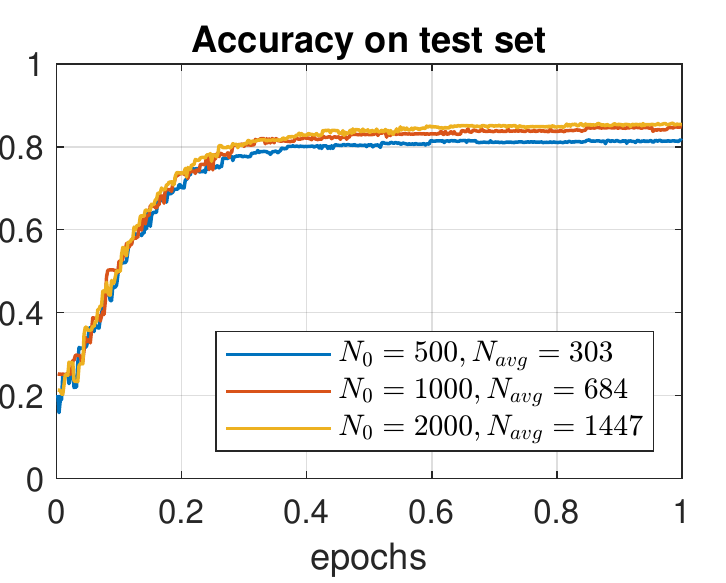}
\includegraphics[trim  = 7cm 0 0 5mm, clip, height = 5.5cm]{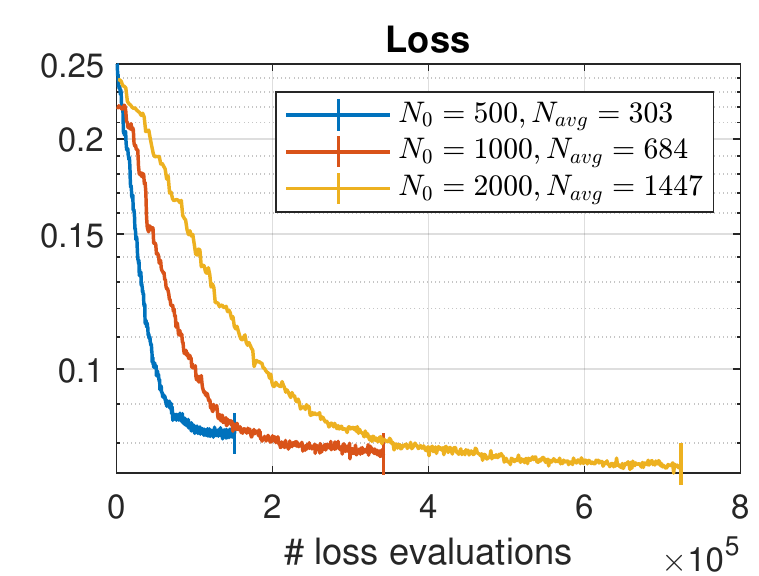} 
\caption{Different initial population sizes $N_0$, same random selection parameter $\mu=0.1$}
\label{fig:MNIST_Nvary}
\end{subfigure}
\caption{Performance of CBO-ME algorithm in training a shallow NN for MNIST classification. Experiment with different combinations of random selection parameter $\mu$ and initial population sizes are considered. Plots on the left display accuracy as a function of the amount of training data considered. On the right, the loss is displayed as a function of the number of loss evaluations. Clearly, when
less particles are employed (either due to large $\mu$ or to small $N_0$), fewer loss evaluations are needed. The average number of particles is denoted by $N_{avg}$. Algorithm parameters are set to $\lambda = 0.1, \sigma = \sqrt{0.1}, \alpha = 5\cdot 10^4$}
\end{figure}

\begin{figure}
\centering
\includegraphics[trim  = 0 0 0 5mm, clip,width = 1\linewidth]{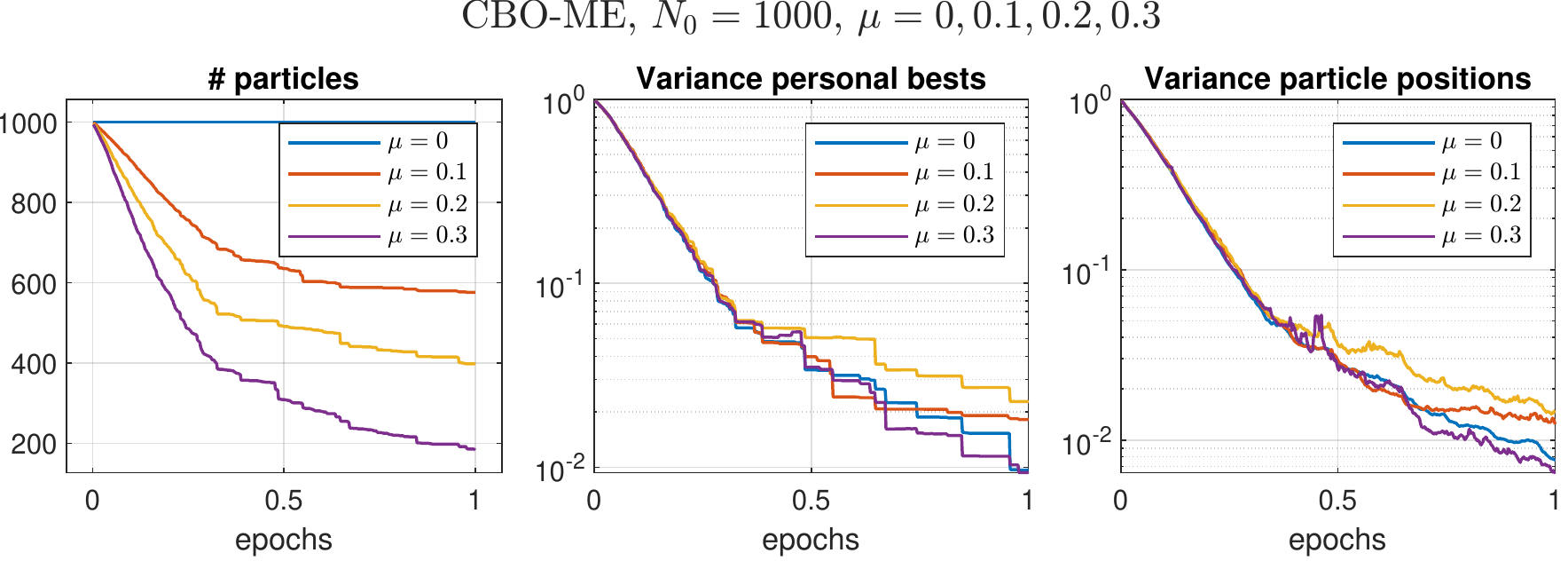} 
\caption{Population statistics during the training of shallow NN with CBO-ME method. Experiments with different random selection parameters $\mu$ and initial population sizes $N_0 = 1000$. Algorithm parameters are set to $\lambda = 0.1, \sigma = \sqrt{0.1}, \alpha = 5\cdot 10^4$}
\label{fig:MNIST_muvary_var}
\end{figure}

\begin{figure}
\centering
\includegraphics[trim  = 0 0 0 5mm, clip, width = 1\linewidth]{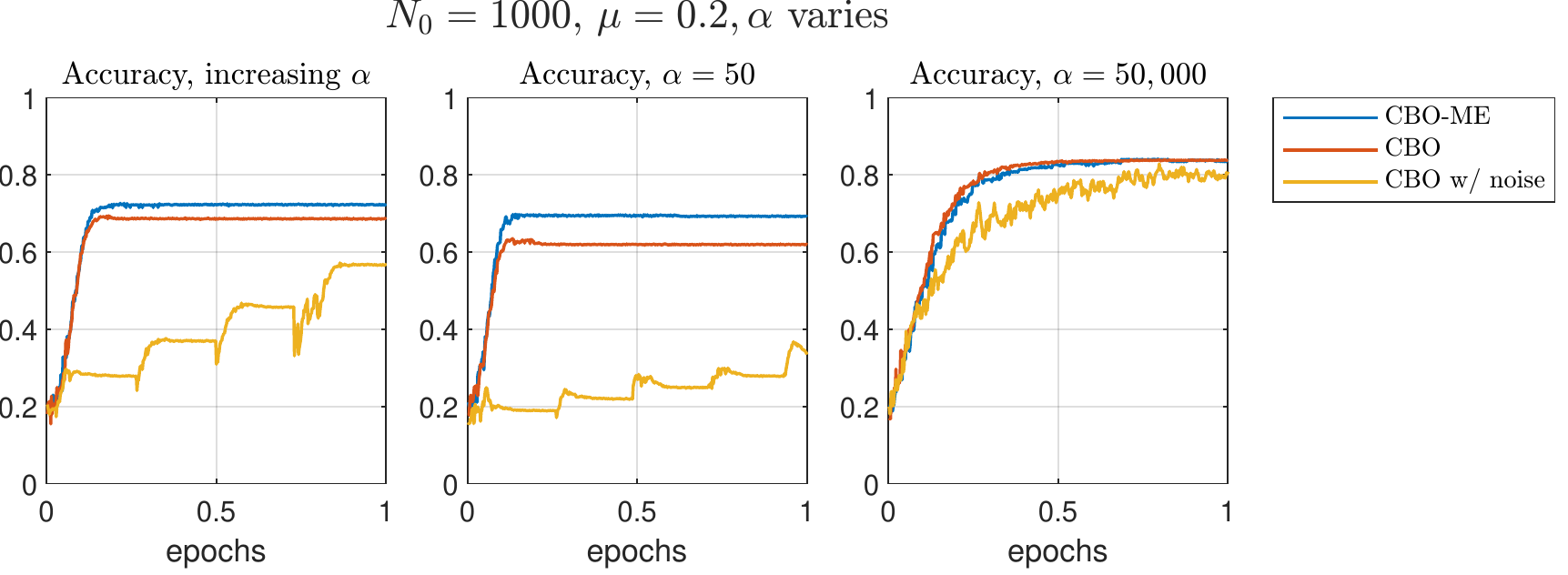} 
\caption{Performance comparison between different consensus algorithms in training of shallow NN. We consider the proposed CBO-ME algorithm, the plain CBO algorithm and the CBO algorithm with random perturbation \eqref{eq:perturbation}, as proposed in \cite{carrillo2019consensus}. For the first two algorithms we set $\sigma = \sqrt{0.1}$, while for the third one $\sigma = \sqrt{0.04}$. We set $\lambda = 0.1$ and consider three different strategies for $\alpha$: $\alpha_k = 5\cdot k, \alpha_k = 50$ and $\alpha_k = 5 \cdot 10^4$}
\label{fig:MNIST_CBO}
\end{figure}

\section{Theoretical analysis}
\label{sec:convergence}
A strength of CBO algorithms lays on the possibility of theoretically analyze the particle system by relying on a mean-field approximation of the dynamics. We will illustrate in this section how to \gb{formally} derive such approximation and present the main theoretical result regarding the convergence of the \gb{mean-field} particle system towards a solution to \eqref{eq:pb}, in case of no selection mechanism.
Next, we will study the impact of the random selection strategy on the convergence properties of the algorithm. Technical details are left to Appendix \ref{app:convergence}.

\subsection{Mean-field approximation}
\label{s:4.1}

First, we note that a simple update rule for the personal bests $y^k_i$ is given by
\be
y^{k+1}_i = y^{k}_i +  \frac12\left(x_i^{k+1} - y_i^k \right)S(x_i^{k+1},y_i^k)\,, \quad \textup{with} \quad S(x,y) = 1 + \textup{sign}\left( \F(y) - \F(x) \right)\,.
\label{eq:updatey}
\ee
As in \cite{grassi2021from}, we approximate it for $\beta \gg 1$ as
\be
y^{k+1}_i = y^{k}_i +  \frac\nu2\left(x_i^{k+1} - y_i^k \right)S^\beta(x_i^{k+1},y_i^k)\,,
\label{eq:updateybeta}
\ee
with $S^\beta(x,y)$ being a continuous approximation of $S(x,y)$ as $\beta \to \infty$. By choosing $\nu=1$ we get \eqref{eq:updatey} with the only difference of having $S^\beta$ instead of $S$. As for $\bar y^{\alpha,k}$ with respect to $\bar y^{\infty,k}$, this is needed to make the update rule easier to handle mathematically, but it does have an impact on the performance for large values of $\beta$. We note that alternative ways of modeling the memory mechanisms have been suggested in the literature of PSO, see, for instance, \cite{mehdi2022pso} where fractional order calculus is used.

With the aim of deriving a continuous-in-time reformulation of \gb{the particle update rules} \eqref{eq:updateCBOME} and \eqref{eq:updateybeta}, we introduce a single parameter $\Delta t>0$ which controls the step length of all involved update mechanisms. By performing the rescaling
\[\lambda \gets \lambda \Delta t\,, \quad  \sigma \gets \sigma \sqrt{\Delta t}\,, \quad \nu \gets \nu \Delta t \]
\gb{we} get the update rules
\be
\begin{cases}
x^{k+1}_i &=  x_i^{k} +  \lambda\Delta t \left( \bar y^{\alpha,k} - x_i^k \right) + \sigma \sqrt{\Delta t} \left( \bar y^{\alpha,k} - x_i^k\right) \otimes \theta^k_i  \\
y^{k+1}_i &=   y^{k}_i +  (\nu\Delta t /2)\left(x_i^{k+1} - y_i^k \right)S^\beta(x_i^{k+1},y_i^k)\, 
\end{cases}
\label{eq:CBOME2}
\ee
which differ from the original formulation \eqref{eq:updateCBOME}, \eqref{eq:updatey} only due to the use of $S^\beta$ instead of $S$. 

As already noted in \cite{grassi2021from}, the iterative process \eqref{eq:CBOME2} corresponds to an Euler-Maruyama scheme applied to a system of Stochastic Differential Equations (SDEs). Indeed, \eqref{eq:CBOME2} corresponds to a discretization of the system
\begin{equation}
\begin{cases}
d X_t^i & = \lambda\left(\bar y^\alpha(\bar \rho_t^N) -  X_t^i\right)\,dt + \sigma \left(\bar y^\alpha( \bar \rho_t^N) -  X_t^i\right)\otimes\,d B_t^i \\
d Y^i_t &= \nu ( X_t^i -  Y_t^i)S^\beta( X_t^i,  Y_t^i)\,dt 
\end{cases}
\label{eq:sde}
\end{equation}
where, for convenience, we underlined above the dependence of the consensus point on the empirical distribution $\bar  \rho_t^N = \sum_i \delta_{Y^i_t}$ ($\delta_y$ being the Dirac measure at $y\in \RR^\d$) by using
\be 
\bar y^\alpha(\rho) := \frac{\int y e^{-\alpha \mathcal{F}(y)} d\rho(y)}{ \int e^{-\alpha \mathcal{F}(y)} d\rho(y)}\,,
\ee
defined for any Borel probability measure $\rho$  over $\RR^\d$ ($\rho \in \mathcal{P}(\RR^\d)$). In this way, we generalized the definition introduced in \eqref{eq:yalpha} to any $\rho \in \mathcal{P}(\RR^\d)$, provided the above integrals exists. 
In \eqref{eq:sde}, the random component of the dynamics is now described by $N$ independent Wiener processes $( B_t^i)_{t>0}$. 
As before, we supplement the system with initial conditions  $ X_0^i \sim \rho_0, Y_0^i = X_0^i$ for some $\rho_0 \in \mathcal{P}(\RR^\d)$.

The continuous-in-time description \eqref{eq:sde} already simplifies the analytical analysis of the optimization algorithm, but still pays the price of a possible large number $\mathcal{O}(N)$ of equations. This issue is typically addressed by assuming that for large populations $N$, particles become indistinguishable from one another and start behaving, in some sense, as a unique system. More precisely, let $F^N(t) \in \mathcal{P}(\RR^{(2\d)N})$ denote the joint probability distribution of $N$ tuples $(X_t^i, Y_t^i)$. We assume \textit{propagation of chaos} \cite{sznitman1991chaos} for large $N \gg1$, that is, we assume that the joint probability distribution decomposes as $F^N(t) = f(t)^{\otimes N}$ for some $f(t) \in \mathcal{P}(\RR^{2\d})$. System \eqref{eq:sde} becomes independent on the index $i$ and hence every particle moves according to the mono-particle process
\be
\begin{cases}
d\bar X_t & = \lambda(\bar y^\alpha(\bar \rho_t) -\bar X_t)\,dt + \sigma\, (\bar y^\alpha(\bar \rho_t) - \bar X_t)\otimes d\bar B_t\\
d\bar Y_t &= \nu (\bar X_t-\bar Y_t)S^\beta(\bar X_t,\bar Y_t)\,dt
\end{cases}
\label{eq:mfsde}
\ee
where $\bar \rho_t = \textup{Law}(\bar Y_t)$.

Assume $(\bar X_t, \bar Y_t)$ are initially distributed according to $f_0 = \rho_0^{\otimes 2}$<. By applying It\^{o}'s formula we have that $f(t) = \textup{Law}(\bar X_t^i, \bar Y_t^i)$ satisfies \gb{in a weak sense}
\be
\partial_t f + \nabla_x \cdot \left( \lambda (\bar y^\alpha (\bar \rho) - x)f \right) + \nabla_y\cdot \left( \nu (x-y) S^\beta(x,y)  f \right) = 
\gb{\frac12 \sum_{\ell=1}^\d \partial_{x_\ell x_\ell}^2 \left (  \sigma^2  (\bar y^\alpha(\bar \rho) - x)_\ell^2 f \right) } 
\label{eq:mfpde}
\ee
with initial data $\lim_{t \to 0} f(t) = f_0$.  Dynamics \eqref{eq:mfsde}, or, equivalently, \eqref{eq:mfpde}, corresponds to the mean-field approximation of the particle system \eqref{eq:sde} as $N \to \infty$. We remark that the above derivation has only been possible thanks to the approximations $S \approx S^\beta$ and $\bar y^\infty \approx \bar y^\alpha$ for large $\alpha$ and $\beta$. Well-posedness of the system is also granted by such approximations, provided the objective function $\F$ satisfies the following assumptions (proof details are given in Appendix \ref{a:well}).
\gb{
\begin{assumption} 
\label{asm:Lipschitz}
The objective function $\F: \RR^\d \to \RR$ is bounded from below, $\inf \F > -\infty$, and there exist some constants  $L_\F, c_u, c_l, R_l>0$ such that
\[
|\F(x)-\F(x')| \leq L_\F\left(\|x\|_2 + \|x'\|_2 \right)\|x-x'\|_2 \qquad \text{for all}\quad x,x'\in\RR^\textup{d}\,,\]
and
\begin{align*}
\F(x)-\inf \F  &\leq c_u(1+\|x\|^2_2) \qquad  &\text{for all}\quad &x \in \RR^\textup{d}\,,\\
\F(x)-\inf \F &  \geq c_l \|x\|_2^2  \qquad  &\text{for all}\quad &x \in \RR^\textup{d}\, \|x\|_2 \geq R_l.
\end{align*}
\end{assumption}
}

\begin{proposition}[\gb{Existence of solution to \eqref{eq:mfsde}}]\gb{Assume $\F$ satisfies Assumption \ref{asm:Lipschitz}.} There exists \gb{a process} $(\bar X, \bar Y) \in C([0,T], \RR^\textup{d})$, $T>0$ satisfying \eqref{eq:sde} with initial conditions $(\bar X_0, \bar Y_0)$ \gb{where} $\bar X_0 \sim \rho_0 \in \mathcal{P}_4(\RR^{\textup{d}})$ and $\bar Y_0 = \bar X_0$.
\label{p:well}
\end{proposition}
Being mathematically tractable, we show next that the mean-field dynamics converges to a global solution to \eqref{eq:pb} \gb{if $\F, S^\beta$ satisfy suitable assumptions}. 


\subsection{Convergence in mean-field law}

We start by enunciating the necessary assumptions to the convergence result.
\begin{assumption} \label{ass:f}
The objective function $\F \in C(\RR^\textup{d},\RR)$, satisfies:
\begin{enumerate}[label=A\arabic*,labelsep=10pt,leftmargin=35pt]
	\item\label{asm:zero_global}
		there exists uniquely $x^*\in \RR^\textup{d}$ solution to \eqref{eq:pb}; 
	\item\label{asm:icp}
		there exist $\eta, R_0>0$ and $\gamma\in\gb{(0,\infty)}$ such that
		\begin{align*}
		\F(x)-\inf \F &\geq \eta\, \|x - x^*\|^{\gamma}_\infty \quad & \forall x \in \RR^{\textup{d}}\,, \;&\|x - x^*\|_{\infty} \leq R_0 \\
		\F(x)-\inf \F& \geq \eta\, R_0^{\gamma} \quad &\forall x \in \RR^\textup{d}\,, \;&\|x-x^*\|_{\infty} > R_0\,.
		\end{align*}
		\item \label{asm:conv}
		$\F$ is convex in a (possibly small) neighborhood $\{x\in \RR^\textup{d}: \|x - x^*\|_{\infty} \leq R_1\}$ of $x^*$ for some $R_1 < R_0$. 

\end{enumerate}
\end{assumption}

\begin{assumption}[Assumptions on $S^\beta$] \label{ass:S}
The function $S^\beta \in C(\RR^{2\textup{d}}, [0,2])$, with $\beta>0$ 

\begin{enumerate}[label=A\arabic*,labelsep=10pt,leftmargin=35pt]
\setcounter{enumi}{3}
\item\label{asm:psi}
has the following structure
\be
S^\beta(x ,y) = 2\psi\left(\beta (\mathcal{F}(y) - \mathcal{F}(x))\right), 
\label{eq:regS}
\ee
with $\psi \in C^{1}(\RR,[0,1])$ being a \gb{non-decreasing function} with Lipschitz constant $L_{\psi}=1$.
\item\label{asm:zero} The value $S^\beta(x,y)$ is positive only when $x$ is strictly better than $y$ in terms of objective value $\F$:
\be \notag
S^\beta(x,y) 
\begin{cases}
 \geq 0 & \text{if}\quad  \mathcal{F}(x) < \mathcal{F}(y) \\
= 0 &\textup{else}\,.
\end{cases}
\ee
\end{enumerate}
\end{assumption}

Assuming uniqueness of global minimum is a typical assumption for analysis of CBO methods \cite{fornasier2021consensusbased,fornasier2022aniso} and it is due to the definition of the consensus point $\bar y^\alpha$ (or $\bar x^\alpha$ in the case without memory mechanism). Indeed, in presence of two global minima, $\bar y^\alpha$ may be placed between them, no matter how large $\alpha$ is. Assumption \ref{asm:zero_global} ensure to avoid such situations. Furthermore, \ref{asm:icp} also allows to give quantitative estimates on the difference between the global minimum and eventual local minima. In the literature, such property is known as $\textit{conditioning}$ \cite{rosasco2017geometry}. Requirements \ref{asm:conv} and \ref{asm:zero} \gb{will be needed to} ensure that if a personal best $y_i^k$ enters such small neighborhood where $\F$ is convex, it will not leave it for the rest of the computation. 
For an intuition of \ref{asm:icp} and \ref{asm:conv} we refer to Figure \ref{fig:illu}, where the Rastrigin function is considered.

\begin{figure}
\centering
\includegraphics[width = 0.45\linewidth]{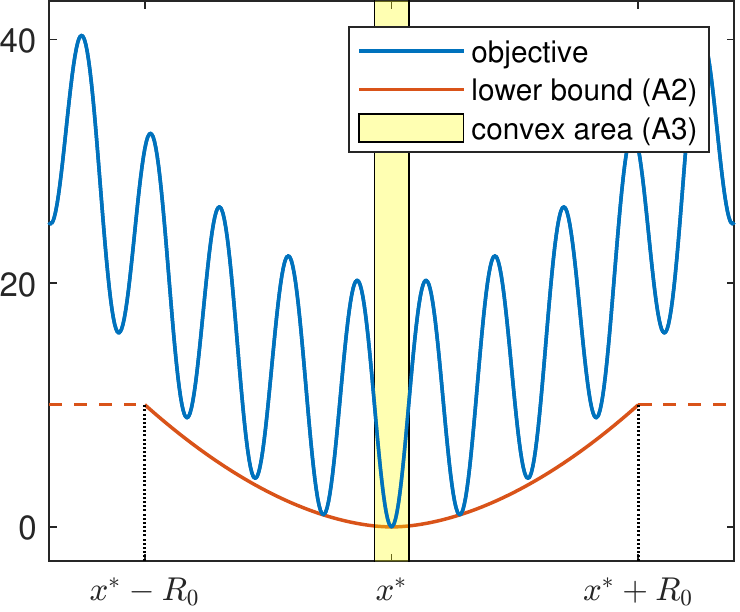}
\caption{Assumptions \ref{ass:f} illustrated for Rastrigin function. For example, such objective function satisfies \ref{asm:icp} with $\eta = 1$, $\gamma = 1.8$, $R_0 = 1.42$ and \ref{asm:conv} with $R_1 = 0.25.$}
\label{fig:illu}
\end{figure}

\begin{theorem}[Convergence in mean-field law] \label{t:main}
Assume $\F$ satisfies Assumption \ref{ass:f} and $S^\beta$ satisfies Assumption \ref{ass:S} for some $\beta>0$ fixed. Let $(\bar X_t, \bar Y_t)_{t\geq0}$ be a solution to \eqref{eq:mfsde} for $t \in [0,T^*]$, with initial data
 $\bar X_0 \sim \rho_0 \in \mathcal{P}_4(\RR^{\textup{d}}), \bar Y_0 = \bar X_0$ such that $x^*\in \textup{supp}(\rho_0)$ .

Fix an accuracy $\varepsilon>0$. If $2\lambda > \sigma^2$, the expected $\ell_2$-error satisfies
\be
\min_{t \in [0,T^*]}\mathbb{E}\left[ \| \bar X_{t} - x^*\|_2^2 \right] \leq  \varepsilon 
\label{eq:errMF}
\ee
provided $T^*, \alpha>0$ are large enough.
\end{theorem}
We refer to Appendix \ref{app:convergence} for a proof.

\begin{remark}
The mean-field mono-particle process \eqref{eq:mfsde} aims to approximate the algorithm iterative dynamics \eqref{eq:CBOME2} for small time steps $\Delta t\ll 1$ and large particle populations $N \gg 1$. Therefore, convergence of the algorithm dynamics towards the global solution $x^*$ can be proven by coupling Theorem \ref{t:main} with error estimates of such approximation. 

For instance, assuming that all considered dynamics take place on a bounded set $\mathcal{D}$ ensures that the error introduced by the continuous-in-time \gb{approximation} will be of order $\Delta t$ thanks to classical results on Euler-Maruyama  schemes \cite{platen1999intro}. Likewise, considering a bounded dynamics allows to prove that the error introduced by the mean-field approximation is of order $N^{-1}$ (see e.g. \cite[Theorem 3.1]{fornasier2020hypersurfaces}, \cite[Proposition 16]{fornasier2021consensusbased}). \gb{If such error rate holds, consider $\{(x^k_i, y^k_i) \}_{i=1}^N$ be given by \eqref{eq:CBOME2}, $\{( X_t^i, Y_t^i) \}_{i=1}^N$ be a solution \eqref{eq:sde} and $\{(\bar X_t^i, \bar Y_t^i) \}_{i=1}^N$ be $N$-copies of a solution to \eqref{eq:mfsde}. Let $t^*\in[0,T^*]$ being a time minimizing the mean-field error in \eqref{eq:errMF}. Altogether, one obtains the following error decomposition for $k = \lfloor t^*/\Delta t\rfloor$, 
\be \notag
\begin{split}
\mathbb{E}\left [ \frac1N \sum_{i=1}^N \|x^{k}_i - x^* \|_2^2 \right] & 
\leq C \Bigg(
 \mathbb{E}\left [ \frac1N \sum_{i=1}^N \|x^{k}_i - X_{t^*}^i \|_2^2 \right] +
 \mathbb{E}\left [ \frac1N \sum_{i=1}^N \|X_{t^*}^i  - \bar X_{t^*}^i  \|_2^2 \right]
 \\
 & \qquad +
 \mathbb{E}\left [ \frac1N \sum_{i=1}^N \|\bar X_{t^*}^i  - x^* \|_2^2 \right]
 \Bigg) \\
 & \leq C_{\textup{EM}}\Delta t  + C_{\textup{MFA}} N^{-1} + \varepsilon
\end{split}
\ee
where $C, C_{\textup{EM}},C_{\textup{MFA}}$ are positive constant independent on $N, \Delta t$.}
\end{remark}

\subsection{Random selection analysis}
\label{s:reductionanalysis}
In this section, we analytically investigate the impact of randomly discarding particles during the computation. We are particularly interested in tracking the distance between 
a particle system \gb{$\{ x_i^k,  y_i^k \}_{i=1}^{N_0}$} evolving according to \eqref{eq:CBOME2} where no particles are discarded, and a second system $\{\hat x_i^k, \hat y_i^k \}_{I_k}, |I_k| = N_k$  where $N_k - N_{k+1}$ particles are discarded after update rule \eqref{eq:CBOME2}. Clearly, we \gb{require} that $N_{k+1} \leq N_k$ and $I_{k+1} \subseteq I_k \subseteq I_0 = \{1, \dots,N_0\}$ for all $k$. Similarly to the analysis carried out in \cite{ha2021timediscrete,jin2020convergence}, we restrict to the simpler dynamics where, at every step $k$, the random variables $\theta^k_i$ and $\hat \theta^k_i$ used to generate such systems are the same for all particles: 
\be
\theta_i^k = \hat\theta^k_j = \theta^k \sim \mathcal{N}(0, \mathbf{I}_{\d}) \quad \textup{for all}\quad i \in I_k,\, j \in I_0.
\label{eq:asm_red}
\ee

To compare particle systems with a different number of particles, we rely on their representation as empirical probability measures and the notion of 2-Wasserstein distance. For $\{\hat x_i^k\}_{i\in I_k}$ and $\{ x_i^k \}_{i =1}^{N_0}$ we consider, respectively, the following probability measures
\be
\rho_{N_k}^k := \frac1{N_k} \sum_{i\in I_k} \delta_{\hat x_i^k} \quad\textup{and} \quad 
\rho_{N_0}^k : = \frac1{N_0} \sum_{i \in I_0}\delta_{ x_i^k}\,.
\label{eq:emp_red}
\ee
Informally, the 2-Wasserstein distance $W_2(\rho_{N_k}^k, \rho_{N_0}^k)$ quantifies the minimal effort needed to move the mass from distribution $\rho_{N_k}^k$ into $\rho_{N_0}^k$  \cite{santambrogio2015optimal}. Let $w_{ij}$ denote the amount of mass leaving particle $x^k_i$ and going into $\hat x^k_i$: the cost of such movement is assumed to be given by $w_{ij}\| x^k_i - \hat x^k_j\|_2^2$. Therefore, if we indicate the set of all admissible couplings between the two discrete probability measures as
\be
\Gamma(\rho_{N_k}^k, \rho_{N_0}^k ) = \left\{ w \in \RR^{N_0 \times N_k}\,:\, 
\sum_{j=1}^{N_k} w_{ij} = \frac{1}{N_0},\,   \sum_{i=1}^{N_0} w_{ij} = \frac{1}{N_k}, \, w_{ij} \geq 0,\, \forall\, i,j  \right \}\,,
\label{eq:wasscoupling}
\ee
the 2- Wasserstein distance is defined as 
\be
W_2(\rho_{N_k}^k, \rho_{N_0}^k) :=  \min_{w \in \Gamma (\rho_{N_k}^k, \rho_{N_0}^k)}\left( \sum_{i,j} w_{ij} \| x_i^k - \hat x_j^k\|_2^2 \right)^{\frac12}
\label{eq:wassemp}
\ee
see, for instance, \cite[Section 6.4.1]{santambrogio2015optimal}.

Before providing estimates on \eqref{eq:wasscoupling}, let us present a more general result on the impact that the random selection strategy has on an arbitrary particle distribution.

\begin{proposition}[Stability of random selection procedure] Let $\mathbf{z} = \{z_i \}_{i \in I}, |I| = N$ be an ensemble of particles and $\{z_i\}_{j \in I_{\red}}$ with $I_{\textup{sel}} \subseteq I, |I| = N_{\textup{sel}}$ a random sub-set of such ensemble. Consider the associated empirical distributions $\mu_N$ and $\mu_{N_\textup{sel}}$ (defined consistently to \eqref{eq:emp_red}). 

It holds
\be
\mathbb{E} \left[W_2^2(\mu_N, \mu_{N_\textup{sel}})\right] \leq 2\,\var (\mathbf{z})\, \frac{N - N_{\textup{sel}}}{N-1}\,,
\label{eq:myestMC}
\ee
where the expectation is taken with respect to the random selection of $I_{\textup{sel}}$.
\label{p:reduction}
\end{proposition}
The proof is provided Appendix \ref{s:proofreduction}. 
We note how the system variance $\var (\mathbf{z})$ enters the error estimate due to the randomness of the selection, similar to the Law of Large Number error for random variables. In particular, the smaller the particles variance is, the closer the reduced particle system will be to the original distribution. This justifies the choice of $N_{k+1}$ proposed in Section \ref{sec:selection} where we are allowed to discard particles only if the system shows a contractive behavior, see \eqref{eq:reduction}.

By iteratively applying Proposition \ref{p:reduction} and by using suitable stability estimates of dynamics \eqref{eq:CBOME2}, we are able to bound the error introduced by the random selection procedure as follows. Proof details are a given in Appendix \ref{s:proofreduction}.

\begin{theorem}  Let $\{ x_i^k,  y_i^k \}_{i=1}^{N_0}$ be constructed according to \eqref{eq:CBOME2} were particles are not discarded, and $\{\hat x_i^k, \hat y_i^k \}_{I_k}, |I_k| = N_k$  where $N_k - N_{k+1}$ particles are discarded after update rule \eqref{eq:CBOME2}. Assume \eqref{eq:asm_red} is satisfied and consider the probability measures \eqref{eq:emp_red}. 

\gb{Under Assumptions \ref{asm:Lipschitz} and \ref{ass:S}},  if $\{ x_i^k,  y_i^k \}_{i=1}^{N_0}, \{\hat x_i^k, \hat y_i^k \}_{i \in I_k} \subset B_M(0)$ at all step $k$ for some $M>0$, it holds
\be
\mathbb{E} \left[W_2^2\left(\rho_{N_k}^k, \rho_{N_0}^k\right) \right] \leq C\, \max_{h = 1, \dots, k} \textup{var}\left(\tilde{\mathbf{z}}^h\right)\; \frac{N_0 - N_k}{N_k -1}
\ee
where $C = C(\Delta t, \lambda, \sigma, \nu, \beta, \alpha, k, L_\F, M)$ and $ \tilde{\mathbf{z}}^{h} = \{ (\hat x_i^h, \hat y_i^h) \}_{i \in I_{h-1}}$ describes the particle system just before the random selection procedure at step $h \leq k$. 
The expectation is taken with respect to the sampling of $\{ \theta^h\}_{h=1}^k$ and with respect to the selection procedure.
\label{t:reduction}
\end{theorem}

We can directly apply the above result to relate the expected $\ell_2$-errors of the two particle system, which we define as
\[ 
\textup{Err}(k) := \mathbb{E}\left[\frac{1}{N_0} \sum_{i \in I_{0}} \|x_i^k - x^*\|^2_2 \right]\,,
\quad 
\textup{Err}_{\red}(k) := \mathbb{E}\left[\frac{1}{N_k} \sum_{i \in I_{k}} \|\hat x_i^k - x^*\|^2_2 \right]\,,
\]
that is, the discrete counterpart of the mean-field error $\mathbb{E}[\|\bar X_t^i - x^*\|_2^2]$ studied in Theorem \ref{t:main}.
By definition of the 2-Wasserstein distance, we have 
\[ 
\textup{Err}(k) = \mathbb{E} \left[W^2_2(\rho_{N_0}^k, \delta_{x^*})\right]
\]
for any solution $x^*$ to \eqref{eq:pb}, and the same holds of $\textup{Err}_{\red}(k)$. We then apply
inequality
\[ W^2_2(\rho_{N_k}^k, \delta_{x^*}) \leq 2 \left( W^2_2(\rho_{N_k}^k,\rho_{N_0}^k) + W^2_2(\rho_{N_0}^k, \delta_{x^*}) \right)\,
\]
to obtain the following estimate.
\begin{corollary} Under the assumptions of Theorem \ref{t:reduction}, at all steps $k$, it holds
\be
\textup{Err}_\textup{sel}(k) \leq 2 \left(\textup{Err}(k) + C \,\max_{h = 1, \dots, k} \textup{var}(\tilde{\mathbf{z}}^h ) \, \frac{N_0 - N_k}{N_k -1} \right )  \,.
\ee
\end{corollary}
Before concluding the section, let us report some remarks concerning the theoretical results just presented. 
\begin{remark}^^>
\begin{itemize}
\item Proof of Theorem \ref{t:reduction} can be adapted to any other particle system with random selection, provided that the update rule is stable with respect to the 2-Wasserstein distance. In the proposed method, such stability was proved thanks to the approximation of the global best $\bar y^{\infty,k}$ with $\bar y^{\alpha,k}$ for $\alpha \gg1$ (see \eqref{eq:yalpha}) and $S(x,y)$ with $S^\beta(x,y)$ for $\beta\gg1$ in the personal best update \eqref{eq:updateybeta}.

\item Quantitative estimates on the variance decay can be used, if available, to improve the error bound in Theorem \ref{t:reduction}, see also proof in Appendix \ref{s:proofreduction}.

\item The error introduced by a sub-sampling technique in a Monte Carlo integral approximation is expected to be of order \be
2\,\var (\mathbf{z})\, \left(\frac{1}{N-1} - \frac{1}{N_{\red}-1}\right) = 2\,\var (\mathbf{z})\,\frac{N - N_\red}{(N-1) (N_\red-1)}\,,
\label{eq:estMC}
\ee
 see e.g. \cite{JLJ}. Therefore, an additional factor of order $1/(N_{\red}-1)$ seems to be missing in Proposition \ref{p:reduction}. We remark, though, that Proposition \ref{p:reduction} does not concern the Monte Carlo approximation of an integral quantity, but rather consider the 2-Wasserstein distance between \gb{empirical probability} measures. 
 Numerical simulations suggest that estimates of order \eqref{eq:estMC} do not hold on in this case, see Fig.\ref{fig:selection}.
 
\end{itemize}
\end{remark}

\begin{figure}
\centering
\hfill
\includegraphics[height = 6cm]{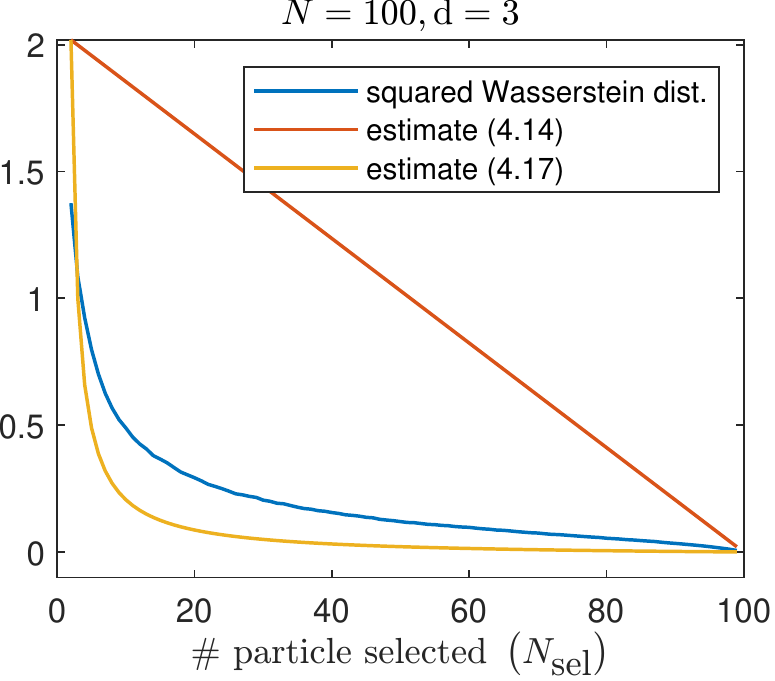}
\hfill
\includegraphics[height = 6cm]{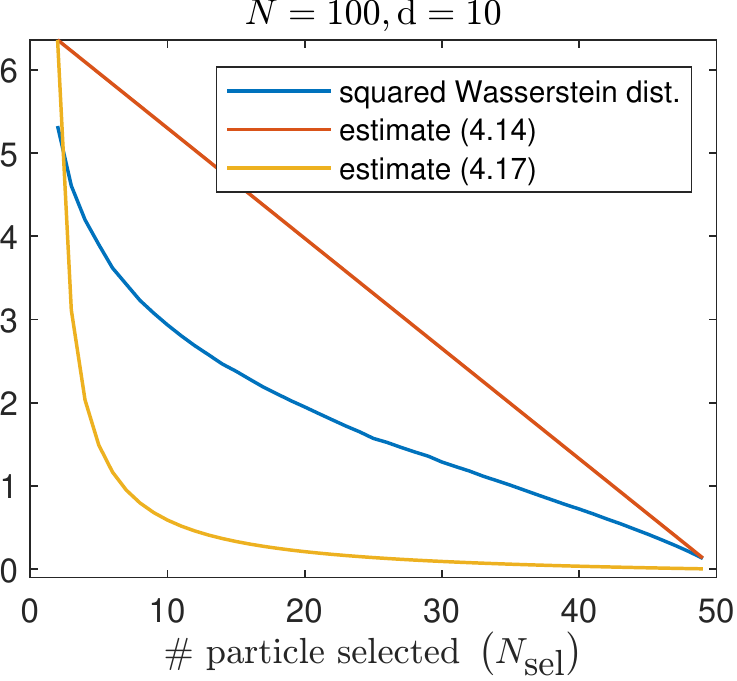}
\hfill
\caption{Numerical validation of Proposition \ref{p:reduction} with different dimensions $\textup{d}=3,10$. $N = 100$ points are randomly, uniformly sampled over $[ 0,1]^\textup{d}$ to construct the empirical distribution $\mu_N$ and $N_{\textup{sel}}\in [2, N-1]$ are discarded to obtain $\mu_{N_\red}$. The experiment is repeated $500$ times for all $N_\red$ to obtain an approximation of $\mathbb{E} \left[W_2^2(\mu_N, \mu_{N_\textup{sel}})\right]$ (blue line). In red, estimate provided by Proposition \ref{p:reduction} (RHS of \eqref{eq:myestMC}), in yellow the one given equation \eqref{eq:estMC}. Wasserstein distances are computed with the \texttt{ot.emd} function provided by the Python Optimal Transport library \cite{flamary2021pot}.
}
\label{fig:selection}
\end{figure}

\section{Conclusions}
\label{sec:conclusions}
In this work, we studied a Consensus-Based Optimization algorithm with Memory Effects (CBO-ME) and random selection for single objective optimization problems of the form \eqref{eq:pb}. While sharing common features with Particle Swarm Optimization (PSO) methods, CBO-ME differs on the way the particle system explore the search space. Its structure provides greater flexibility in balancing the exploration and exploitation processes. In particular, we implemented and analytically \gb{investigatesd} a random selection strategy which allows to reduce the algorithm computational complexity, without affecting convergence properties and overall accuracy. This analysis is entirely general and, in perspective, applicable to other \gb{particle-based} optimization methods as well. The convergence analysis to the global minimum is carried out by relying on a mean-field approximation of the particle system and error estimates are given under mild assumptions on the objective function. We compared CBO-ME against CBO without memory effects and PSO against several benchmark problem and showed how the introduction of memory effects and random selection improves the algorithm performance. Applications to image segmentation and machine learning problems are \gb{also} reported. 

\appendix

\section{Proofs}
\label{app:convergence}



%
\subsection{Notation and auxiliary lemmas}

We will use the following notation.  For any $a \in \RR$, $|a|$ indicates the absolute value. For a given vector $b\in \RR^\d$,  $\| b\|_{p}$ indicates its $p$-norm, $p \in [1,\infty]$; $(b)_\ell$ its $\ell$-th component; while $\textup{diag}(b) \in \RR^{\d \times \d}$ is the diagonal matrix with elements of $b$ on the main diagonal. Let $a, b \in \RR^\d, \langle a, b \rangle$ denotes the scalar product in $\RR^\d$.
For a given closed convex set $A \subset\RR^\d$, $\mathcal{N}(A,x), \mathcal{T}(A,x)$ denote the \gb{Clarke} normal and the tangential cone at $x\in A$ respectively. 
\gb{The $\ell^p$-ball, $p \in [1, \infty]$, of radius $r$ centered at $x\in \RR^\d$ is indicated with $B^p_r(x)=\{x\in \RR^\d \,|\, \|x\|_p \leq r\}$.} 
All considered stochastic processes are assumed to take their realizations over the common probability space
$(\Omega, \mathcal{\bar F}, \mathbb{P})$. $\mathcal{P}(\RR^\d)$ is the set of Borel probability measures over $\RR^\d$ and $\mathcal{P}_q (\RR^\d) = \{\mu \in \mathcal{P}(\RR^\d) \,|\, \int \|x\|_2^q d\mu<\infty\}$ which we equip with the Wasserstein distance $W_q$, $q \geq 1$, see \cite{santambrogio2015optimal}. For a random variable $X$, $X \sim \mu$, $\mu\in \mathcal{P}(\RR^\d)$ indicates a sampling procedure such that $\mathbb{P}(X \in A ) = \mu(A)$ for any Borel set $A \subset \RR^\d$. With $\textup{Unif}(A) \in \mathcal{P}(\RR^\d)$ we denote the uniform probability measure over a bounded Borel set $A$. 
Throughout the computations, $C$ will denote an arbitrary positive constant, whose value may vary from line to line. Dependence on relevant parameters or variables, will be underlined.

\begin{lemma}[{\cite[Lemma 3.2]{carrillo2018analytical}}]
\gb{Let $\mathcal{F}$ satisfy Assumption \ref{asm:Lipschitz}} and 
 $\rho_1, \rho_2 \in \mathcal{P}_4(\RR^\d)$ with
\[ \int \|x\|_2^4 \,d\rho_1\,, \quad \int \|x\|_2^4\, d\rho_2 \leq M\,.
\]
Then, the following stability estimate holds
\be
\| \bar y^\alpha(\rho_1) - \bar y^\alpha(\rho_2) \|_2 \leq C\, W_2(\rho_1, \rho_2)
\label{eq:est1} \notag
\ee 
for a constant $C = C(\alpha, L_\F, M)$.
\label{l:yalpha}
\end{lemma}

\begin{lemma} Under Assumptions \ref{asm:Lipschitz} and \ref{ass:S}, for any $x_1, x_2, y_1, y_2 \in \gb{B^2_M}(0)$, it holds
\[
\|(x_1 - y_1)S^\beta(x_1, y_1) - (x_2 - y_2)S^\beta(x_2,y_2) \|_2 \leq C\, \left(\|x_1 - y_1\|_2 + \|x_2 - y_2 \|_2 \right)
\]
where $C = C(\beta, L_\F, M)$.
\label{l:sbeta}
\end{lemma}
\begin{proof}\gb{ We note that function $S^\beta$ is locally Lipschitz continuous thanks to the locally Lipschitz continuity of $\F$ (Assumption \ref{asm:Lipschitz}) and the Lipschitz continuity of $\psi$ (Assumption \ref{ass:S}):} 
\begin{align*}
| S^\beta(x_1, y_1) - S^\beta(x_2,y_2) | &= |2 \psi \left(\beta(\F(y_1) - \F(x_1))  -2 \psi (\beta(\F(y_2) - \F(x_2) \right)| \\
& \leq 2\beta \left | \F(y_1) - \F(x_1) -  \F(y_2) + \F(x_2)  \right| \\
& \leq \gb{4 \beta L_\F M }\left( \|x_1 - x_2\|_2 + \|y_1 - y_2\|_2  \right)\,.
\end{align*}
Next, we have
\begin{align*}
\|(x_1 - y_1)S^\beta(x_1,& y_1) - (x_2 - y_2)S^\beta(x_2,y_2) \|_2  \\
 &\leq \|(x_1 - y_1)S^\beta(x_1, y_1)- (x_2 - y_2)S^\beta(x_1, y_1)\|_2\\
 &\gb{ \qquad +\|(x_2 - y_2)S^\beta(x_1, y_1) - (x_2 - y_2)S^\beta(x_2,y_2) \|_2} \\
 &\leq \| (x_1 - x_2 +y_2 - y_1 ) S^\beta(x_1, y_1)\|_2  + \|(x_2 - y_2) \left(S^\beta(x_1, y_1) - S^\beta(x_2,y_2) \right) \|_2 \\
 &\leq 2 \left (\|x_1 - x_2\|_2 + \|y_1 - y_2\|_2\right) + 2M| S^\beta(x_1, y_1) - S^\beta(x_2,y_2)| \\
 &\leq C \left (\|x_1 - x_2\|_2 + \|y_1 - y_2\|_2\right) 
\end{align*}
with $C = C(\beta, L_\F, M)$, where we used the 
\gb{proved locally Lipischitz continuity of $S^\beta$ in the last inequality.}
\end{proof}

\subsection{Proof of Proposition \ref{p:well}}
\label{a:well}
\begin{proof}[Proof of Proposition \ref{p:well}] 
The proof is based on the Leray–Schauder fixed point theorem \cite[Chapter 11]{gilbarg2001elliptic}, and we follow the proof of \cite[Theorem 3.2]{carrillo2018analytical}.

\textbf{Step 1.}
For any $\xi \in C([0,T], \RR^\textup{d})$ there exists a unique process $(\hat X_t, \hat Y_t) \in C([0,T], \RR^\d)$ satisfying
\be\notag
\begin{split}
d\hat X_t & = \lambda(\xi(t) -\hat X_t)\,dt + \sigma (\xi(t) - \hat X_t)\otimes\,d\hat B_t\\
d\hat Y_t &= \nu (\hat X_t-\hat Y_t)S^\beta(\hat X_t,\hat Y_t)\,dt 
\end{split}
\ee
with $\textup{Law}(\hat X_0) = \textup{Law}(\hat Y_0)  = \rho_0 \in \mathcal{P}(\RR^\d)$, \gb{by  locally Lipschitz continuity and linear growth of the coefficients (thanks to Lemma \ref{l:sbeta})}. As a consequence, we also have that $f(t):= \textup{Law}(\hat X_t,\hat Y_t)$ satisfies 
\[\gb{
\frac d{dt} \int \phi\,df(t) = \int \left( -\lambda \langle\nabla_x \phi, \xi(t) - x \rangle + \frac 12 \sigma \sum_{\ell=1}^{\textup{d}} \frac{\partial^2 \phi}{\partial x_\ell^2}(\xi (t) - y)_{\ell}^2 
- \nu S^\beta \langle \nabla_y \phi, y-x \rangle \right) df(t)}
\]
for all $\phi \in C_b^2(\RR^{2\d})$ by applying It{\^o}'s formula. Therefore, let $\bar \rho(t) = \textup{Law}(\hat Y_t)$, we can set $\mathcal{T} \xi:= \bar y^\alpha(\bar\rho(\cdot)) \in C([0,T], \RR^{\d})$ to define
\be \notag
\mathcal{T}: C([0,T], \RR^{\d}) \to C([0,T], \RR^{\d}) .
\ee

\textbf{Step 2.} We prove now compactness of $\mathcal{T}$.  Thanks to $\rho_0 \in \mathcal{P}_4 (\RR^{\d})$ and standard results for SDEs (see \cite[Chapter 7]{arnold1974stochastic}) we have boundedness of the forth moments
\be \notag
\mathbb{E}\left[ \| \hat X_t\|_2^4 + \|\hat Y_t \|_2^4 \right] \leq c_1\left( 1 + \mathbb{E}[   \|\hat X_0\|_2^4 + \|\hat Y_0\|_2^4 ] e^{c_2 t} \right)
\ee
for some $c_1, c_2>0$. Therefore, we can apply Lemma \ref{l:yalpha} to obtain for any $0<s<t<T$, 
\be \notag
\|\bar  y^\alpha(\bar \rho(t)) - \bar y^\alpha(\bar \rho(s))\|_2 \leq C W_2\left(\bar \rho(t), \bar \rho(s)\right) \leq \tilde C |t - s|^{1/2}\,
\ee 
for some constants $C, \tilde C>0$, from which Hölder continuity of $\,t \mapsto \bar y^\alpha(\bar \rho(t))$ follows. 
Compactness of $\mathcal{T}$ follows by
\[\mathcal{T}( C([0,T], \RR^{\d})) \subset C^{0,\frac12}([0,T], \RR^{\d})\hookrightarrow C([0,T], \RR^{\d})\,.\]

\textbf{Step 3.}
Consider $\xi \in C([0,T],\RR^\d)$ satisfying $\xi = \tau \mathcal{T} \xi $, for $\tau \in [0,1].$ Thanks to \cite[Lemma 3.3]{carrillo2018analytical} and boundedness of second moments, we obtain 
compactness of the set 
\[ \{\xi \in C([0,T], \RR^\d)\,:\, \xi = \tau \mathcal{T} \xi, \tau \in[0,1] \}  \] 
and by Leray–Schauder fixed point theorem there exists a fixed point for the mapping $\mathcal{T}$ and hence a solution to \eqref{eq:mfsde}.

\end{proof}

\subsection{Proof of Theorem \ref{t:main}}

Having proved there exists a solution $(\bar X_t,\bar  Y_t)_{t\in [0,T]}$ to the mean-field process \eqref{eq:mfsde}
we are here interested in studying the expected $\ell_2$-error given by 
\be
\mathbb{E}\|\bar X_t - x^* \|^2_2 \notag
\ee
where $x^*$ is the unique solution to the minimization problem \eqref{eq:pb} \gb{(uniqueness is given by Assumption \ref{ass:f})}. We do so by means of the following quantitative version of the Laplace principle. 

\begin{proposition}[\gb{Quantitative} Laplace principle {\cite[Proposition 1]{fornasier2022aniso}}]

Let $\rho \in \mathcal{P}(\RR^\textup{d})$ be such that $x^* \in \textup{supp}( \rho)$ and fix $\alpha>0$. For any $r>0$, define $\F_r := \sup_{x \in B^\infty_r(x^*)} \F(x) - \F(x^*)$\,.

Then, under Assumption \ref{ass:f}, for any $ r \in (0,R_0]$ and $q>0$ such that $q+ \mathcal{F}_r \leq \F_{\infty} := \eta R_0^\gamma$, it holds
\be
\| \gb{\bar y^\alpha( \rho)} - x^* \|_2 \leq \frac{\sqrt{\textup{d}}(q + \mathcal{F}_r)^{\gamma}}\eta + \frac{\sqrt{\textup{d}}\exp (-\alpha q)}{ \rho (B^\infty_r(x^*))} \int \| x - x^*\|_2\, d \rho(x).
\label{eq:laplace}
\ee
\label{p:laplace}
\end{proposition}

We remark that RHS of \eqref{eq:laplace} can be made arbitrary small by taking large values of $\alpha$ and small values of $q,r$ \gb{provided the integral is bounded.}

\gb{\begin{lemma}Let $(\bar X_t ,\bar Y_t)_{t \in [0,T]}$ be a solution to \eqref{eq:mfsde} and initial data $\bar X_0  = \bar Y_0$ and $x^* \in \RR^\textup{d}$. For any $t\in[0,T]$, it holds
\be
\mathbb{E}[\|\bar Y_t - x^* \|^2_2] \leq 2  e^{\nu t}\sup_{s \in [0,t]}\mathbb{E}[\|\bar X_s - x^* \|_2^2]\,.
\label{eq:lemma3}
\ee
\end{lemma}
\begin{proof}
Due to \eqref{eq:mfsde} \gb{and chain rule}, it holds
\begin{align*}
\frac{d}{dt} \|\bar Y_t - x^* \|^2_2  &= 2 \nu \langle \bar Y_t - x^* , \bar X_t - \bar Y_t \rangle S^\beta(\bar X_t, \bar Y_t) dt \\
& = 2\nu \langle \bar Y_t - x^* , \bar X_t - x^*\rangle  S^\beta(\bar X_t, \bar Y_t) dt - 2\nu \|\bar Y_t - x^* \|_2^2  S^\beta(\bar X_t, \bar Y_t) dt \\
& \leq  \nu \left( \|\bar Y_t - x^*\|_2^2 + \|\bar X_t - x^*\|^2_2 \right)dt
\end{align*}
By taking the expectation and applying Grönwall's inequality, we have
\[
 \mathbb{E}[\|\bar Y_t - x^* \|^2_2]  \leq  \mathbb{E}[\|\bar Y_0 - x^* \|^2_2] e^{ \nu t} + \int_0^t \mathbb{E}[\|\bar X_s - x^* \|^2_2] e^{\nu (t-s)} ds\,. 
\]
Estimate \eqref{eq:lemma3} can be obtained after noting that $\mathbb{E}[\|\bar Y_0 - x^* \|^2_2]  = \mathbb{E}[\|\bar X_0 - x^* \|^2_2]$ due to choice of the initial data, and by taking the supremum over all times $s \in [0,t]$.
\end{proof}
}

To apply Proposition \ref{p:laplace} to all $\bar \rho(t) = \textup{Law}(\bar Y_t),$ we need though to provide lower bounds on $\gb{\bar \rho(t) (B^\infty_r(x^*))}$ for any small radius $r$ and times $t\in [0,T]$.

\begin{lemma} Let $\bar \rho(t) = \textup{Law}(\bar Y_t)$, with $\bar Y_t$ evolving according to \eqref{eq:mfsde} and $\lim_{t \to 0}\bar \rho(t)  = \rho_0$ with $x^* \in \textup{supp}(\rho_0)$. Under Assumptions \ref{ass:f} and \ref{ass:S} ,
it holds $\bar \rho(t) (B^\infty_r(x^*)) \geq m_r>0$, for all $t \in [0,T]$ and for all $r\leq R_0$.
\label{l:mass}
\end{lemma}

\begin{proof}
Let $\delta = \eta \min \{R_1, r\}^{\gamma}$, we start by proving that the mass in the set
\be
L_\delta = \{x \in \RR^\textup{d}\; |\; \mathcal{F}(x)  \leq \inf \mathcal{F} + \delta  \}
\notag
\ee
is non-decreasing. We note that for this choice of $\delta$, $L_\delta$ is convex due to Assumption \ref{ass:f}. Consider now $(\Omega, \mathcal{\bar F}, \mathbb{P})$ to be the common probability space over which the considered processes take their realization and define $\Omega_\delta = \{\omega \,:\, \bar Y_0(\omega) \in L_\delta \}$. By Assumption \ref{ass:S}, $S^\beta(\bar X_t(\omega), \bar Y_t(\omega)) = 0$ whenever $\bar X_t(\omega) \notin L_\delta$. Therefore, it holds
\be
\left \langle (\bar X_t(\omega) - \bar Y_t(\omega))S^\beta(\bar X_t(\omega), \bar Y_t(\omega)) \,,\,  n(\bar Y_t(\omega)) \right \rangle 
\begin{cases}
=  0 & \textup{if}\;\bar X_t(\omega) \notin L_\delta \\
\leq 0 & \textup{if}\; \bar X_t(\omega) \in L_\delta
\end{cases}
\quad \text{for} \quad \bar Y_t(\omega) \in \partial L_\delta
\notag
\ee
for any $n(\bar Y_t(\omega)) \in \mathcal{N}(L_\delta,x)$   from which follows that $\bar Y_t(\omega)$ solves
\be
\bar Y_t(\omega) = \bar Y_0(\omega) + \int_0^t \Pi_{\mathcal{T}(L_\delta, \bar Y_s(\omega))}\left (   (\bar X_s(\omega) - \bar Y_s(\omega))S^\beta(\bar X_s(\omega), \bar Y_s(\omega))  \right) ds\,
\notag
\ee
for all $\omega \in \Omega_{\delta}$. As a consequence, if $\bar Y_0(\omega) \in L_\delta$, $\bar Y_t(\omega) \in L_\delta$ for all $t \geq0$ and so 
\[\bar \rho(t) (B^\infty_r(x^*)) = \mathbb{P}(\bar Y_t \in L_\delta )  \geq  \mathbb{P}(\bar Y_0 \in L_\delta)=:m_r\]
 for all $t \geq 0$. We conclude by noting that $m_r>0$ since $x^* \in \textup{supp}(\rho_0)$.
\end{proof}

Next, we study the evolution of the error $\mathbb{E}\| \bar X_t - x^*\|_2^2$ and, in particular, we try to bound it in terms of  $\|\bar y^\alpha(\bar \rho(s)) - x^*\|_2$ and $\mathbb{E}\| \bar X_t - x^*\|_2$ itself for $s \in [0,t]$.

\begin{proposition}\cite[Lemma 1]{fornasier2022aniso}
Let $(\bar X_t, \bar Y_t) \in C([0,T], \RR^{2\textup{d}})$ satisfy \eqref{eq:mfsde} with initial datum $\bar X_0 \sim \rho_0, \rho_0 \in \mathcal{P}_4(\RR^\textup{d}), \bar Y_0 = \bar X_0$ for some time horizon $T>0$. 

Set $\mathcal{V}(\rho(t)):=(1/2)\mathbb{E}\| \bar X_t - x^* \|_2^2$ with $\rho(t) \in \textup{Law}(\bar X_t)$.
For all $t \in [0,T]$, it holds

\begin{multline}
\frac{d}{dt} \mathcal{V}(\rho(t))  \leq  -(2\lambda - \sigma^2)\mathcal{V}(\rho(t)) + \sqrt{2}(\lambda + \sigma^2) \sqrt{\mathcal{V}(\rho(t))}\, \|\bar y^\alpha(\bar \rho(t)) - x^*\|_2 \\+ \frac{\sigma^2}2  \|\bar y^\alpha(\bar \rho(t)) - x^*\|_2^2 \,.
\label{eq:errevo}
\end{multline}

%
where $\bar \rho(t) = \textup{Law}(\bar Y_t)$.
\label{p:errevolution}
\end{proposition}

\begin{proof}[Proof of Theorem \ref{t:main}]
The above result, together with Lemma \ref{l:mass}, leads to the convergence in mean-field law of the dynamics towards the solution to \eqref{eq:pb}. 
The proof can be carried out exactly as in \cite[Theorem 12]{fornasier2022aniso} and we summarize here the main steps for completeness. 

For notational simplicity, we introduce $\err(t):=\mathbb{E}[\|\bar X_t - x^*\|_2^2]$. We start by setting the time horizon $T^* = -2\log(\ve/\err(0)) / (2\lambda - \sigma^2)$. We apply Proposition \ref{p:errevolution} and, since $\mathcal{V}(\rho(t)) = \err(t)/2$, we have for all $t \in [0,T^*]$
\[
\frac{d}{dt} \err(t) \leq -(2\lambda - \sigma^2) \err(t) + (\lambda + \sigma^2) \sqrt{\err(t)} \|\bar y^\alpha(\bar \rho(t)) - x^*\|_2 + \sigma^2 \|\bar y^\alpha(\bar \rho(t)) - x^*\|_2^2\,.
\]
Let $T\geq0$ be given by
\[
T:= \sup \left\{t \in [0,T^*] \,:\, \err(t') > \ve \;\textup{and}\;  \|\bar y^\alpha(\bar \rho(t')) - x^*\|_2^2 < C(t') \quad \forall t'\in[0,t] \right\}
\]
with 
\[
C(t) := \min \left \{  \frac 14 \frac{2\lambda - \sigma^2}{\lambda + \sigma^2}\,,\, \sqrt{\frac14 \frac{2\lambda - \sigma^2}{\sigma^2}}\right \} \sqrt{\err(t)}\,.
\]
For this particular choice of $T$, we have that  for all $t \in [0,T]$
\be
\frac{d}{dt} \err(t) \leq -\frac 12(2\lambda - \sigma^2) \err(t) \quad \Rightarrow \quad \err(t) \leq \err(0)\, e^{-\frac12 (2\lambda - \sigma^2)t}
\label{eq:decayE}
\ee
where we applied Grönwall's inequality. Now, we consider three possible scenarios.

\textbf{Case $T = T^*$.} By definition of $T^*$ and decay estimate \eqref{eq:decayE}, we have $\err(T^*) = \ve$.

\textbf{Case $T < T^*$ and} $\err(T) = \ve$. Nothing to prove in this case.

\textbf{Case $T < T^*$ and $\|\bar y^\alpha(\bar \rho(s)) - x^*\|_2^2 \geq C(T)$}. We will show that if $\alpha$ is large enough, this case cannot occur. From Proposition \ref{p:laplace} and Lemma \ref{l:mass} we have
\[ 
\| \gb{\bar y^\alpha( \bar \rho(T))} - x^* \|_2 \leq \frac{\sqrt{\textup{d}}(q + \mathcal{F}_r)^{\gamma}}\eta + \frac{\sqrt{\textup{d}}\exp (-\alpha q)}{m_r} \mathbb{E}[\| \bar Y_T - x^* \|_2] 
\]
Now, by continuity of $\mathcal F$, we can take $q, r$ small small enough such that the first term on the right-hand side is strictly smaller than $C(T)/2$. Thanks to Lemma \ref{eq:lemma3} and bound \eqref{eq:decayE}, it holds
\[
\mathbb{E}[\| \bar Y_T - x^* \|_2] \leq \left(\mathbb{E}[\| \bar Y_T - x^* \|^2_2]  \right)^{1/2}
\leq \sqrt{2}e^{\frac12 \nu T} (\sup_{t \in [0,T]} \err(t)) ^{1/2} \leq \sqrt{2}e^{\frac12 \nu T} \sqrt{\err(0)}\,. 
\]
Therefore, we can take $\alpha$ sufficiently large such that 
\be
 \frac{\sqrt{\textup{d}}\exp (-\alpha q)}{m_r} \mathbb{E}[\| \bar Y_T - x^* \|_2]  \leq  \frac{\sqrt{\textup{d}}\exp (-\alpha q)}{m_r}\sqrt{2}e^{\frac12 \nu T} \sqrt{\err(0)} < \frac{C(T)}2\, ,
\ee
from which follows 
\[
\| \gb{\bar y^\alpha( \bar \rho(T))} - x^* \|_2 < C(T)\,.
\]
Therefore, we have a contradiction and we can conclude that this third case can be avoided by taking $\alpha$ sufficiently large.
\end{proof}


\subsection{Proof of Proposition \ref{p:reduction} and Theorem \ref{t:reduction}}
\label{s:proofreduction}
We start by collecting a preliminary result.

\begin{lemma}
Let $\{x_{1,i}^k,y_{1,i}^k \}_{i=1}^{N_1}$ and $\{x_{2,j}^k,y_{2,j}^k \}_{j=1}^{N_2}$ be two particle populations generated through update rules \eqref{eq:CBOME2}
with $\theta_{1,i}^k = \theta_{2,j}^k = \theta^k$ for all $i,j$ and $k \in \mathbb{Z}_+$. At any iteration step $k$ and for any couple of indexes  $(i,j)$, it holds
\begin{multline}
\mathbb{E}\left [ \| x^{k+1}_{1,i} - x^{k+1}_{2,j}\|_2^2 + \| y^{k+1}_{1,i} - y^{k+1}_{2,j}\|_2^2 \right] \leq \\ 
C \mathbb{E}\left [\| x^{k}_{1,i} - x^{k}_{2,j}\|_2^2 + \| y^{k}_{1,i} - y^{k}_{2,j}\|_2^2  + \|\bar y^\alpha(\bar \rho_1^k) - \bar y^\alpha(\bar \rho_2^k) \|_2^2 \right] \notag
\end{multline}
where $C = C(\Delta t, \lambda, \sigma, \nu, \beta)$ is a positive constant and $\bar \rho^k_1,\bar \rho^k_2$ are the empiricial distributions associated with $\{y_{1,i}^{k} \}_{i=1}^{N_1}$ and $\{y_{2,j}^{k} \}_{j=1}^{N_2}$ respectively.
\label{l:step}
\end{lemma}

\begin{proof}
For all $k \in \mathbb{Z}_+$ and $i,j$
\begin{align}\notag
\mathbb{E} \|x^{k+1}_{1,i} - x^{k+1}_{2,j}\|_2^2 &  \leq\mathbb{E} \Big \| x_{1,i}^k + \lambda\Delta t \left( \bar y^{\alpha}(\bar \rho^k_{1}) - x_{1,i}^k \right) + \sigma \sqrt{\Delta t} \left( \bar y^{\alpha}(\bar \rho^k_{1}) - x_{1,i}^k\right) \otimes \theta^k_{1,i} \\ \notag
 & \qquad -  \left(  x_{2,j}^k + \lambda\Delta t \left( \bar y^{\alpha}(\bar \rho^k_{2}) - x_{2,j}^k \right) + \sigma \sqrt{\Delta t} \left( \bar y^{\alpha}(\bar \rho^k_{2}) - x_{2,j}^k\right) \otimes \theta^k_{2,j} \right) \Big \|_2^2 \\ \notag
 & \leq 2 \mathbb{E}\left \| \left(1 - \lambda \Delta t - \sigma \sqrt{\Delta t}\,\textup{diag}(\theta^k) \right) (x_{1,i}^k - x_{2,j}^k)\right\|_2^2 \\ \notag
 &\qquad + 2\mathbb{E}\left \| \left( \lambda \Delta t + \sigma \sqrt{\Delta t}\,\textup{diag}(\theta^k)  \right) \left (\bar y^{\alpha}(\bar \rho^k_{1}) -  \bar y^{\alpha}(\bar \rho^k_{2})\right) \right\|_2^2\\
 & \leq 2(1+\lambda^2 \Delta t^2+ \sigma^2 \Delta t)\mathbb{E}\| x_{1,i}^k - x_{2,j}^k\|_2^2 \notag \\
 & \qquad + 2(\lambda^2 \Delta t^2 + \sigma^2 \Delta t) \mathbb{E}\|\bar y^{\alpha}(\bar \rho^k_{1}) -  \bar y^{\alpha}(\bar \rho^k_{2}) \|_2^2\,,
\label{eq:stepestimatex}
\end{align}
where we also used that $\mathbb{E}[(\theta^k)_\ell^2] = 1$ for all $\ell = 1, \dots, \d$.
We now bound $\|y_{1,i}^{k+1} - y_{2,j}^{k+1}\|_2^2$ as
\begin{align} \notag
\mathbb{E}\|y_{1,i}^{k+1} - y_{2,j}^{k+1}\|_2^2 &\leq \mathbb{E} \Big \|  y^{k}_{1,i} +  (\nu\Delta t /2)\left(x_{i,1}^{k+1} - y_{1,i}^k \right)S^\beta(x_{1,i}^{k+1},y_{1,i}^k)  \\ \notag
& \qquad - \left(y^{k}_{2,j} +  (\nu\Delta t /2)\left(x_{2,j}^{k+1} - y_{2,j}^k \right)S^\beta(x_{2,j}^{k+1},y_{2,j}^k) \right)   \Big    \|_2^2 \\ 
& \leq C \mathbb{E} \left [\|x_{i,1}^{k+1}  - x_{j,2}^{k+1}\|_2^2 + \|y_{i,1}^{k}  - y_{j,2}^{k}\|_2^2 \right]
\label{eq:stepestimatey}
\end{align}
where we used Lemma \ref{l:sbeta} and $C = C(\Delta t, \beta, \nu)$. By combining \eqref{eq:stepestimatex} and \eqref{eq:stepestimatey} we get the desired estimate.
\end{proof}

Next, we show how the particle update rule \eqref{eq:CBOME2} is stable with respect to the 2-Wasserstein distance.

\begin{proposition}[Stability of update rule \eqref{eq:CBOME2}] Let $\{x_{1,i}^k,y_{1,i}^k \}_{i=1}^{N_1}, \{x_{2,j}^k,y_{2,j}^k \}_{j=1}^{N_2} \subset B_M(0)$, for some $M>0$, be two particle populations generated through the update rules \eqref{eq:CBOME2}
with $\theta_{1,i}^k = \theta_{2,j}^k = \theta^k$ for all $i,j$ and $k \in \mathbb{Z}_+$.
Let $\mu^k_1,\mu_2^k \in \mathcal{P}(\RR^{2\textup{d}})$ the empirical probability measures defined as
\[
\mu_1^k := \frac1{N_1}\sum_{i=1}^{N_1} \delta_{(x_{1,i}^k, y_{1,i}^k)}\,,\quad 
\mu_2^k := \frac1{N_2}\sum_{j=1}^{N_2} \delta_{(x_{2,j}^k, y_{2,j}^k)}\,.
\]

\gb{Under Assumptions \ref{asm:Lipschitz} and \ref{ass:S},}
it holds
\[
\mathbb{E}\left[ W_2^2(\mu^{k+1}_1, \mu^{k+1}_2) \right] \leq C_1\, \mathbb{E}\left[W_2^2(\mu_1^k, \mu_2^k)\right]\,,
\]
where $C_1 = C_1(\Delta, \lambda, \sigma, \nu, \alpha, \beta, L_\F, M)$ is positive constant.
\label{p:wass_stability}
\end{proposition}

\begin{proof} 
Let $\mathbb{E}_{\theta^k}[\cdot]$ denote the expectation taken with respect to the sampling of $\theta^k$ only and $w \in \RR^{N_1 \times N_2}$ be the optimal coupling between $\mu_1^k, \mu_2^k$ (see \eqref{eq:wasscoupling} and \eqref{eq:wassemp}). Being $w$ a sub-optimal coupling for $\mu_1^{k+1}, \mu_2^{k+1}$, it holds
\begin{align*}
\mathbb{E}_{\theta^k}[W_2^2(\mu_1^{k+1},\mu_2^{k+1}) ] &\leq \mathbb{E}_{\theta^k} \sum_{i,j} w_{ij} \left( \|x^{k+1}_{1,i} - x_{2,j}^{k+1} \|_2^2 + \|y^{k+1}_{1,i} - y_{2,j}^{k+1} \|_2^2  \right) \\
& \leq C \sum_{i,j} w_{ij} \left( \|x^k_{1,i} - x_{2,j}^k \|_2^2 + \|y^k_{1,i} - y_{2,j}^k \|_2^2  \right) + \|\bar y^{\alpha}(\bar \rho^k_{1}) -  \bar y^{\alpha}(\bar \rho^k_{2}) \|_2^2
\end{align*}
where we used the linearity of the expectation, estimates given by Lemma \ref{l:step} and,  to take the last term out of the sum, the fact that $\sum_{ij} w_{ij} =1$. 

To estimate the distance between the two consensus points, we use Lemma \ref{l:yalpha} and note that the coupling $w$ is sub-optimal for $\bar\rho ^k_1,\bar \rho^k_2$ with respect to the optimal transport. By Lemma \ref{l:yalpha}, it follows
\[
\|\bar y^{\alpha}(\bar \rho^k_{1}) - \bar  y^{\alpha}(\bar \rho^k_{2}) \|_2^2 \leq \;C W_2^2(\bar \rho^k_{1}, \bar \rho^k_{2}) \;\leq C\sum_{i,j} w_{ij}  \|y^k_{1,i} - y_{2,j}^k \|^2\,.
\]
Therefore, 
\begin{align*}
\mathbb{E}_{\theta^k}[W_2^2(\mu_1^{k+1},\mu_2^{k+1}) ] &\leq C_1 \sum_{i,j} w_{ij} \left( \|x^k_{1,i} - x_{2,j}^k \|_2^2 + \|y^k_{1,i} - y_{2,j}^k \|_2^2  \right) = C_1 \,W_2^2(\mu_1^k, \mu_2^k)\,, 
\end{align*}
thanks to the optimality of $w$, with $C_1 =  C_1(\Delta, \lambda, \sigma, \nu, \alpha, \beta,L_\F,M)$ being a positive constant.  One can  conclude by taking the expectation of the above inequality with respect to the sampling of $\theta^h, h<k$.
\end{proof}

We now quantify the impact of the particle discarding step.

\begin{proof}[Proof of Proposition \ref{p:reduction}] For notational simplicity, let us introduce $z_i = (x_i, y_i) \in \mathbb{R}^{2\d}$. As in \eqref{eq:wassemp}, the 2-Wasserstein distance is given by an optimal coupling between the full particle system $\{z_i\}_{i\in I}$ and the reduced one $\{z_j\}_{j \in I_{\red}}$. We consider the following transportation of mass from $\mu_N$ to $\mu_{N_{\red}}$: if particle $i$ has not been discarded, all its mass remains in $x_i$, otherwise the mass is uniformly distributed among the selected particles to generate an admissible coupling $w\in \mathbb{R}^{N \times N_{\red}}$. This means that $w$ is given by
\be \quad  w_{ij} = \begin{cases}
1/N & \textup{if}\; j = i, i \in I_{\red}\\
1/(N \cdot N_{\red})& \textup{if}\;  i \in I \setminus I_{\red}, j \in I_{\red}\\
0& \textup{else}\,.
\end{cases}
\label{eq:wchoice}
\ee
We note that such coupling $w$ satisfies the coupling conditions
\be \sum_{j \in I_{\red}} w_{ij} = \frac1N\, \quad \sum_{i\in I}w_{ij} = \frac1{N_\red}, \quad \forall\; i \in I, j \in I_{\red}\,
\label{eq:wcoupling}
\ee
and that this choice will be in general sub-optimal. Therefore, it holds
\begin{align*}
W_2^2( \mu_N, \mu_{N_\red}) &\leq \sum_{i\in I,\, j \in I_\red} w_{ij} \| z_i - z_j\|_2^2 \\
&= \frac1N \sum_{i \in I_\red}  \|z_i - z_i \|_2^2 + \frac1{N \cdot N_{\red}} \sum_{i \in I\setminus I_\red, \, j \in I_\red} \|z_i - z_j \|_2^2 \\
& = \frac1{N \cdot N_{\red}} \sum_{i,j \in I} \|z_i - z_j \|_2^2\, \mathbf{1}_{i \in I\setminus I_\red}\,\mathbf{1}_{j \in I_{\red}}
\end{align*}
where $\mathbf{1}_{i \in A} = 1$ if $i \in A$ and $\mathbf{1}_{i \in A} = 0$ if $i \notin A$. 

Now, the probability of having $i \in I \setminus I_{\red}$ is given by  $(N - N_{\red})/N$, while the probability of having $j \in I_{\red}$ (condition $i \in I \setminus I_{\red}$) is given by $N_{\red}/(N-1)$. Hence, we have
\[
\mathbb{E} \left[\mathbf{1}_{i \in I\setminus I_\red}\,\mathbf{1}_{j \in I_{\red}} \right] = \mathbb{P}\left[i \in I \setminus I_{\red}, j \in I_{\red} \right] = \frac{(N - N_\red)N_{\red}}{N (N- 1)}\,,
\]
from which follows
\begin{align*}
\mathbb{E}\left [W_2^2( \mu_N, \mu_{N_\red}) \right] &\leq
\frac1{N \cdot N_{\red}} \sum_{i,j \in I} \|z_i - z_j \|_2^2\,  \mathbb{E} \left[\mathbf{1}_{i \in I\setminus I_\red}\,\mathbf{1}_{j \in I_{\red}}
 \right] \\
 &  = \frac1{N \cdot N_{\red}} \cdot \frac{(N - N_\red)N_{\red}}{N (N- 1)} \sum_{i,j \in I} \|z_i - z_j \|^2_2\,.
\end{align*}
The desired estimates is then obtained by noting that the variance can be computed as $\var(\mathbf{z}) = 1/(2N^2)\sum_{i,j\in I} \|z_i - z_j \|_2^2$, see definition \eqref{eq:variance}.

\end{proof}

Finally, we are ready to provide a proof of Theorem \ref{t:reduction}.

\begin{proof}[Proof of Theorem \ref{t:reduction}]
Let $\{(x_i^k, y_i^k)\}_{i \in I_k}, |I_k| = N_k$ be the sequence of particles generated by iteration \eqref{eq:CBOME2} where additionally $N_{k+1}-N_k$ particles are discarded after each step $k \geq 0$. We denote with $\mu^k_{N_k}\in \mathcal{P}(\RR^{2\textup{d}})$ the empirical measure associated with such particle system given by
\be\notag
\mu^k_{N_k} = \frac1{N_k}\sum_{i\in I_k} \delta_{(x_i^k,y_i^k)}\,.
\ee
We also introduce the measures $\mu^k_{N_0}, k \geq 0$ corresponding to a particle system generated with the same initial conditions $\mu_{N_0}^0$ but where no particle reduction occurs. Consistently, we define $\mu_{N_k}^h$, $h > k$ to represent the particle system generated starting from $\mu_{N_k}^k$, after $h-k$ iterations, with no random selection. The relation between such measures is summarized in the following diagram
\be
\begin{matrix}
\mu_{N_0}^0 & \rightarrow & \mu_{N_0}^1       & \rightarrow  & \mu_{N_0}^2     & \rightarrow & \dots        &\rightarrow & \mu_{N_0}^k   & {\color{white} \vdots }            \\
	             &                   & \dashdownarrow   &                   &                          &                   &	          &                  &                  	&	\\
	             &                   & \mu_{N_1}^1      & \rightarrow & \mu_{N_1}^2      & \rightarrow  & \dots        &\rightarrow & \mu_{N_1}^k    &  {\color{white} \vdots }  \\
	              &                   &	                         &                   & \dashdownarrow &                  &		  &                   &                         & 	 \\
	              &                   &	                         &                   & \mu_{N_2}^2     & \rightarrow & \dots        &   \rightarrow                &   \mu_{N_2}^k       & {\color{white} \vdots } \\
                      &                   &	                         &                   &                          &                    &               &                       &                      & \\
                      &                   &	                         &                   &                          &                    &\ddots      &                       &               \vdots        & \\
                      &                   &	                         &                   &                          &                    &                 &                       &                       & \\
                       &                   &	                         &                   &                          &                   &                 &                   & \mu_{N_{k}}^k & {\color{white} \vdots } \\
\end{matrix}
\label{eq:scheme}
\ee
where $\rightarrow$ indicates an iteration step \eqref{eq:CBOME2} while  $\dashrightarrow$ a particle reduction procedure. Therefore, we are interested in studying the distance between the main diagonal of such diagram $\mu_{N_k}^k$, corresponding to the system with particle reduction, and the first row $\mu_{N_0}^k$ where particle reduction is never performed. 

We note that the 2-Wasserstein distance between subsequent rows can be estimated thanks to Proposition \ref{p:wass_stability} and Proposition \ref{p:reduction}. Let $\tilde{\mathbf{z}}^{h+1}$ denote the set of particles associated with the probability measure $\mu^{h+1}_{N_h}$, that is, the particle systems before the selection procedure (upper diagonal elements in scheme \eqref{eq:scheme}). By first applying Proposition \ref{p:wass_stability} and, subsequently, Proposition \ref{p:reduction} to $\tilde{\mathbf{z}}^{h+1}$, we obtain that for some constant $C>0$
\begin{align*}
\mathbb{E} \left[W_2^2(\mu_{N_k}^k, \mu_{N_0}^k)\right] &\leq C \,\sum_{h = 0}^{k-1} \mathbb{E} \left[ W_2^2 \left ( \mu_{N_h}^k, \mu_{N_{h+1}}^k\right)\right ] \\
& \leq C\, \sum_{h=0}^{k-1} C_1^{k-h+1}\, \mathbb{E} \left[W_2^2\left(\mu_{N_h}^{h+1}, \mu_{N_{h+1}}^{h+1}\right)\right] \\
& \leq 2\,C\, \sum_{h=0}^{k-1} C_1^{k-h+1}\, \var \left(\tilde{\mathbf{z}}^{h+1}\right) 
\frac{N_h - N_{h +1}}{N_h-1}\\
& \leq C_2\,\max_{h = 1, \dots, k} \var \left(\tilde{\mathbf{z}}^{h}\right)\frac{1}{N_k-1} \,\sum_{h=0}^{k-1} N_h - N_{h+1} \\
& = C_2\, \max_{h = 1, \dots, k} \var \left(\tilde{\mathbf{z}}^{h}\right)  \frac{N_0 - N_k}{N_k-1}
\end{align*}
with $C_2 = C_2(\Delta t, \lambda, \sigma, \nu, \beta, \alpha, k, M)$. Finally, the desired estimate follows after noting that 
\[ W_2^2(\rho^{k}_{N_k}, \rho^k_{N_0}) \leq W_2^2(\mu_{N_k}^k, \mu_{N_0}^k)
\]
since $\|x_i^k - x^k_j \|_2^2 \leq \|(x_i^k, y^k_i) - (x_j^k , y_j^k)\|_2^2$ for all couples of particles $(i,j)$.
\end{proof}

\section*{Acknowledgments}
This work has been written within the activities of GNCS group of INdAM (National Institute of High Mathematics). L.P. acknowledges the partial support of MIUR-PRIN Project 2017, No. 2017KKJP4X “Innovative numerical methods for evolutionary partial differential equations and applications”. The work of G.B. is funded by the Deutsche Forschungsgemeinschaft (DFG, German Research Foundation) through 320021702/GRK2326 ``Energy, Entropy, and Dissipative Dynamics (EDDy)'' and SFB 1481 ``Sparsity and Singular Structures''. S.G. acknowledges the support of the ESF PhD Grant “Mathematical and statistical methods for machine learning in biomedical and socio-sanitary applications”.

\bibliographystyle{abbrv}
\bibliography{bibfile}

\end{document}